\documentclass[a4paper]{article}
\usepackage[all]{xy}\usepackage[latin1]{inputenc}        %accents
\usepackage[dvips]{graphics,graphicx}
\usepackage{amsfonts,amssymb,amsmath,color,mathrsfs, amstext}
\usepackage{amsbsy, amsopn, amscd, amsxtra, amsthm,authblk}
\usepackage{enumerate,algorithmic,algorithm}
\usepackage{upref}
\usepackage{geometry}
\usepackage[displaymath]{lineno}
\usepackage{float}
\usepackage{yhmath}
\usepackage{booktabs}
\usepackage{subcaption}
\usepackage{multirow}
\usepackage{makecell}
\usepackage[colorlinks,
            linkcolor=red,
            anchorcolor=red,
            citecolor=red
            ]{hyperref}

\numberwithin{equation}{section}

\def\e{\epsilon}

\def\R{\mathbb{R}}
\def\T{\mathbb{T}}

\def\cD{\mathcal{D}}
\def\cI{\mathcal{I}}

\def\loc{\mathrm{loc}}

\def\pc{\bar{*}}

\newtheorem{theorem}{Theorem}[section]
\newtheorem{lemma}{Lemma}[section]
\newtheorem{proposition}{Proposition}[section]
\newtheorem{remark}{Remark}[section]
\newtheorem{corollary}{Corollary}[section]
\newtheorem{definition}{Definition}[section]
\newtheorem{assumption}{Assumption}[section]

\makeatletter
\@mparswitchfalse%Put all sidenotes on the right.
\makeatother

\title{Some Gr\"onwall inequalities for a class of discretizations of time fractional equations on nonuniform meshes}

\author[a]{Yuanyuan Feng\thanks{
E-mail: yyfeng@math.ecnu.edu.cn}}
\author[b]{Lei Li\thanks{E-mail: leili2010@sjtu.edu.cn}}
\author[c]{Jian-Guo Liu\thanks{E-mail: jliu@math.duke.edu}}
\author[d]{Tao Tang\thanks{E-mail:  ttang@uic.edu.cn}}
\affil[a]{School of  Mathematical Sciences, Key Laboratory of MEA(Ministry of Education) \& Shanghai Key Laboratory of PMMP, East China Normal University, Shanghai, 200241, China. }
\affil[b]{School of Mathematical Sciences, Institute of Natural Sciences, MOE-LSC, Shanghai Jiao Tong University, Shanghai, 200240, China.}
\affil[c]{Department of Mathematics, Department of Physics, Duke University, Durham, NC 27708, USA.}
\affil[d]{Division of Science and Technology, BNU-HKBU United International College, Zhuhai, 519087, China}

\date{}

\begin{document}

\maketitle

\begin{abstract}
We consider the completely positive discretizations of fractional ordinary differential equations (FODEs) on nonuniform meshes. Making use of the resolvents for nonuniform meshes, we first establish comparison principles for the discretizations. Then we prove some discrete Gr\"onwall inequalities using the comparison principles and careful analysis of the solutions to the time continuous FODEs. Our results do not have any restrictions on the step size ratio. The Gr\"onwall inequalities for dissipative equations can be used to obtain the uniform-in-time error control and decay estimates of the numerical solutions.  The Gr\"onwall inequalities are then applied to subdiffusion problems and the time fractional Allen-Cahn equations for illustration.
\end{abstract}

\section{Introduction}

The time fractional differential equations with Caputo derivatives  \cite{caputo1967linear,diethelm10}  have been widely used to model the power law memory effects of energy dissipation for some anelastic materials, and soon became a useful modeling tool  in engineering and physical sciences to construct physical models for nonlocal interactions in time (see \cite{zaslavsky2002chaos}). The models with Caputo derivatives may also result from the complexity reduction and the generalized Langevin dynamics \cite{kouxie04,li2017fractional}. The Caputo derivatives are more suitable for studying the initial value problems as it removes the singularity in the Riemann-Liouville derivatives \cite{liliu18frac}. 
In analyzing these models, especially the a priori energy estimates for time fractional PDEs, it is crucial to have the comparison principles and Gr\"onwall inequalities for the time fractional ordinary differential equations (FODEs) that take values in $\R$.

The FODE taking values in $\R$ with Caputo derivative of order $\alpha\in (0, 1)$ can be written as
\begin{gather}\label{eq:fracode}
D_c^{\alpha}u=f(t, u), \quad u(0)=u_0,
\end{gather}
where $f: [0,\infty) \times \R\to \R$ is assumed to be locally Lipschitz and $u: [0,T)\to \R$ for some $T>0$ is the unknown function. 
Here, the Caputo derivative is defined by
\begin{gather}\label{eq:captra}
D_c^{\alpha}u=\frac{1}{\Gamma(1-\alpha)}\int_0^t\frac{u'(s)}{(t-s)^{\alpha}}ds,
\end{gather}
where $\Gamma(\cdot)$ is the gamma function.
Often the comparison principles may involve functions $u$ that are continuous but not absolutely continuous  so the generalized definitions of Caputo derivatives in \cite{liliu18frac,liliu2018compact} might be considered in these cases. One may refer to section \ref{subsec:capintro} for more details.  The FODE \eqref{eq:fracode} is equivalent to the integral equation (see \cite{diethelm10} and also \cite{liliu18frac,liliu2018compact} for generalized versions)
\begin{gather}\label{eq:fracint}
u(t)=u_0+\frac{1}{\Gamma(\alpha)}\int_0^{t}(t-s)^{\alpha-1}f(s,u(s))ds.
\end{gather}
In other words, the FODE is equivalent to a Volterra equation with the Abel integral kernel
\begin{gather}\label{eq:kernelfode}
g_{\alpha}(t):=\frac{1}{\Gamma(\alpha)}t_+^{\alpha-1},
\end{gather}
where $t_+=t\theta(t)$ and $\theta(t)=1_{t\ge 0}$ is the standard Heaviside step function. 
The kernel $g_{\alpha}$ is known to be completely monotone \cite{widder41,ssv12}, and thus log-convex. A function $a: (0,\infty)\to \R$ is said to be completely monotone if $(-1)^n a^{(n)}(t)\ge 0$ for $n=0,1,\cdots$. A nonzero completely monotone function is strictly positive by the Bernstein theorem \cite{widder41,ssv12}.
Using the integral formulation, it is clear that as $\alpha \equiv 1$, \eqref{eq:fracode} reduces to the usual ODE. 

The first type of comparison principles is for two solution curves of \eqref{eq:fracode} or \eqref{eq:fracint}.  If the initial value of one solution is smaller, the solution is always smaller. Such results are well-established and one may refer to \cite{fllx18note,feng2023class} for examples. A more useful type of comparison principles is for inequalities, which can give estimates to some energy functionals and norms for the solutions. Such results using the differential inequalities are actually also well-known \cite{ramirez2012generalized,fllx18note} but the proofs there are not easy to generalize to numerical schemes. For uniform meshes, the comparison principles for the differential form have been established in \cite{li2019discretization} and \cite{li2021complete}. The proof there, however, heavily relies on the discretization and is intrinsically different from the proof for the time continuous version in \cite{ramirez2012generalized,fllx18note}.  It is thus desired that the proof can be motivated from the analysis for the time continuous equations.  Besides, due to the weak singularity in the memory kernels, the FODE models often exhibit multi-scale behaviors and the solution has singulairty at $t=0$ \cite{cuesta2006convolution,stynes2017error,tang2019energy}. The adaptive time-stepping is often adopted to address this issue \cite{mclean1996discretization,liao2019discrete,stynes2017error,li2019linearized,
 lyu2022symmetric}. Hence, it is desired to establish the comparison principles for the variable step discretizations.

Establishing explicit bounds for the discrete time fractional inequalities, or the discrete Gr\"onwall inequalities, is of great significance to get a priori bounds and to prove stability of numerical schemes. The Gr\"onwall inequalites with linear function $f(u)=\beta u+c$ may be established using the comparison principles because the solutions to the FODEs with linear $f(\cdot)$ are explicitly  known.  There are several results about the discrete Gr\"onwall inequalities in literature \cite{feng2018continuous,liao2018sharp,liao2019discrete}.  The Gr\"onwall inequality in \cite{feng2018continuous} is restricted to uniform meshes. The ones in \cite{liao2018sharp} are based on the special form of L1 scheme so they only apply to L1 discretizations. The ones in \cite{liao2019discrete} can apply to a broader class but there is step ratio restriction and are not directly applicable to estimate the decaying rate of solutions. Often the error control for dissipative systems like the subdiffusion problems rely on the maximum principles \cite{stynes2017error} so it cannot get the decay bounds for the numerical solutions.

In this work, we aim to establish the comparison principles for completely positive discretizations on nonuniform meshes (detailed in section \ref{sec:mondis}). There is no monotonicity assumption on the function $f$. Our approach is to discover a new proof for the comparison principles for time continuous FODEs based on the resolvents \cite{clement1981asymptotic,miller1968volterra}, using similar techniques as in \cite{fllx18note,fengli2023a}. With the so-called pseudo-convolution (see the details in section \ref{subsec:nonunidis}), one can define the resolvent kernels for nonuniform meshes as well. Then, we can establish the comparison principles for the variable-step discretizations using the resolvent kernels by generalizing the argument for the continuous case. This is our new approach.  Based on the comparison principles, we establish Gr\"onwall inequalities for the discretizations making use of some key properties of the solutions to continuous equations. The main results can be summarized as below.
\begin{theorem}[Informal version of Theorem \ref{thm:unibound} and \ref{thm:decayest}]
Consider the differential form \eqref{eq:diffdis} that is completely positive and assume that $c_{n-j}^n$ is comparible to the average of $g_{1-\alpha}$ on the $j$th interval in some sense (see the corresponding sections for the details). 
If a nonnegative sequence $v_n$ satisfies that $\cD_{\tau}^{\alpha}v_n\le -\lambda v_n+c$ for $\lambda>0$, then for some constants $\nu$ and $\tilde{\sigma}$, it holds that
\begin{enumerate}[(a)]
\item $v_n \le v_0 E_{\alpha}(-\nu^{-1}\lambda t_n^{\alpha})+(c/\lambda)(1-E_{\alpha}(-\nu^{-1}\lambda t_n^{\alpha}))$ if $v_0 \le c/\lambda$;
\item $v_n\le (v_0-c/\lambda)E_{\alpha}(-\tilde{\sigma}\lambda t_n^{\alpha})+c/\lambda$ if $v_0>c/\lambda$.
\end{enumerate}
\end{theorem}
For the two cases, the final statements are similar. We divide them into $v_0\le c/\lambda$ and $v_0>c/\lambda$ because the proofs rely on different properties of the solutions to the continuous equations. The result in (a) is useful for the uniform bound estimate of the numerical solutions while the result in (b) is useful to get the decay rate of the solutions. One may obtain the decay estimates of norms of the numerical solutions and the uniform-in-time error estimates for dissipative systems based on these results.
Another result is the following.
\begin{theorem}[Informal version of Theorem \ref{thm:grongrow}]
Consider the differential form \eqref{eq:diffdis}  that is completely positive. Assume that $c_{n-j}^n$ is comparible to the average of $g_{1-\alpha}$ on the $j$th interval and the stepsize $\tau_n$ satisfies that $\lambda \tau_n^{\alpha}\le \delta$ for some $\delta>0$  (see the corresponding section for the details). If $\cD_{\tau}^{\alpha}v_n\le \lambda v_n+c$ for $\lambda>0$, then it holds for some $\mu>0$ that $y_n \le (y_0+c/\lambda)E_{\alpha}(\mu^{-1}\lambda t_n^{\alpha})-c/\lambda$.
\end{theorem}
In this result, we have removed the usual constraint on the stepsize ratio in literature. The result is based on the comparison principles and a careful analysis of the asymptotic behavior of the solutions to the continuous equation.

The rest of the paper is organized as follows. In section \ref{sec:prelim}, some preliminary concepts and results are reviewed, including the definition of generalized Caputo derivatives, the behaviors of the Mittag-Leffler functions and the discretization we conisder in this work on nonuniform meshes. In section \ref{sec:comparisonprinciple}, we present a new proof of comparision principles using the resolvent kernels for the time continuous equations and then generalize it to completely positive discretizations on nonuniform meshes.  Our main results for Gr\"onwall inequalities are then established in section \ref{sec:gronwall} and some applications to dissipative systems are presented in section \ref{sec:application}.

\section{Preliminaries and setup}\label{sec:prelim}

In this section, we review some basic concepts and results for later sections.

\subsection{The generalized Caputo derivatives and the resolvent kernels}\label{subsec:capintro}

In this subsection, we summarize the generalized Caputo derivative introduced in \cite{liliu18frac,liliu2018compact}. This generalized definition is theoretically convenient, since it allows us to use the underlying group structure. Moreover, the generalzied definition allows us to consider the fractional {\it inequalities} for functions that are merely continuous.

Recall the Abel integral kernels for $\alpha>0$ in \eqref{eq:kernelfode}. Recall the standard one-sided convolution for two functions $u$ and $v$ defined on $[0,\infty)$
\begin{gather}
u*v(t)=\int_{[0, t]}u(s)v(t-s)\,ds,
\end{gather}
which can be generalized to distributions whose supports are on $[0,\infty)$ (see \cite[sections 2.1,2.2]{liliu18frac}). Let $g_0=\delta$ and 
\begin{gather}
g_{\beta}(t)=\frac{1}{\Gamma(1+\beta)}D\left(\theta(t)t^{\beta}\right), \beta\in (-1, 0).
\end{gather} 
Here $D$ means the distributional derivative on $\mathbb{R}$.  Then, for any $\beta_1>-1$ and $\beta_2>-1$,
\begin{gather}\label{eq:group}
g_{\beta_1}*g_{\beta_2}=g_{\beta_1+\beta_2}.
\end{gather}
We remark that $g_{\beta}$ can indeed be defined for $\beta\in \mathbb{R}$ (see \cite{liliu18frac}) so that $\{g_{\beta}: \beta\in\mathbb{R}\}$ forms a convolution group. Then, we introduce the generalized definition.
\begin{definition}[\cite{liliu18frac,liliu2018compact}] \label{def:caputo}
Let $0<\alpha<1$ and $T>0$. For $u\in L_{\loc}^1[0, T)$ and a given $u_0\in \mathbb{R}$,the $\alpha$-th order generalized Caputo derivative of $u$ associated with initial value $u_0$ is a distribution with support in $[0, T)$ given by 
\begin{gather}\label{eq:capgen}
D_c^{\alpha}u=g_{-\alpha}*\Big((u-u_0)\theta(t)\Big).
\end{gather}
\end{definition}
It has been verified in \cite{liliu18frac} that if the function $u$ is absolutely continuous, the generalized definition reduces to the classical definition \eqref{eq:captra}.    A function $u\in L_{\loc}^1[0, T)$ is a weak solution to \eqref{eq:fracode} on $[0, T)$ with initial value $u_0$ if the equality holds in distribution. A weak solution $u$ is a strong solution if  $\lim_{t\to 0+}\frac{1}{t}\int_0^t|u(s)-u_0|ds=0$ and both sides of   \eqref{eq:fracode} are locally integrable on $[0,T)$. 
Using the group property \eqref{eq:group}, one may obtain directly that
\begin{proposition}[{\cite[Proposition 4.2]{liliu18frac}}]\label{pro:equi}
Suppose $f\in L_{\loc}^{\infty}([0,\infty)\times\mathbb{R}; \mathbb{R})$. Fix $T>0$. Then, $u(t)\in L_{\loc}^1[0, T)$ with initial value $u_0$ is a strong solution of \eqref{eq:fracode} on $(0, T)$ if and only if it solves the following integral equation 
\begin{gather}\label{eq:vol}
u(t)=u_0+\frac{1}{\Gamma(\alpha)}\int_0^t(t-s)^{\alpha-1}f(s, u(s))ds,~\forall t\in (0, T).
\end{gather}
\end{proposition}

Using this integral formulation, the following has been shown in \cite{liliu18frac}. 
\begin{lemma}\label{pro:exisunique}
Suppose  $f:  [0,\infty)\times (u_*, u^*)\to \mathbb{R}$ is continuous and locally Lipschitz in $u$ where $u_*\in [-\infty,\infty)$ and $u^*\in (-\infty, \infty]$. For any $u_0\in (u_*, u^*)$, there is a maximum time $T_b>0$ and a unique weak solution on $[0, T_b)$ satisfying $u(0+)=u_0$. This weak solution is a strong solution, and if $T_b<\infty$ then either $\limsup_{t\to T_b^-}u(t)=u^*$ or $\liminf_{t\to T_b^-}u(t)=u_*$.
\end{lemma}
Since the weak solution becomes the strong solution and in fact $C^1(0,T_b)\cap C[0, T_b)$, the Caputo derivative then reduces to the classical one.   Consider the following linear FODE
\begin{gather}
D_c^{\alpha}v=\beta v+c,
\end{gather}
the solution exists globally (i.e., $T_b=\infty$) and is given by
\begin{gather}\label{eq:exactsol}
v=\left(v_0+\frac{c}{\beta}\right)E_{\alpha}(\beta t^{\alpha})-\frac{c}{\beta}.
\end{gather}
Here, $E_{\alpha}(z):=E_{\alpha, 1}(z)$ and the Mittag-Leffler function $E_{\alpha,\beta}$ is an entire function given by (see, for example, \cite{mainardi2007probability,gllm02})
\begin{gather}\label{eq:powerseriesofml}
E_{\alpha, \beta}(z)=\sum_{k=0}^{\infty}\frac{z^k}{\Gamma(\alpha k+\beta)}, \alpha>0, z\in 
\mathbb{C}.
\end{gather}
It is then clear that $E_{\alpha}'(z)=\frac{E_{\alpha,0}(z)}{\alpha z}=\alpha^{-1}E_{\alpha, \alpha}(z)$.
The function $E_{\alpha,\beta}(z)$ has the following integral representation (see \cite{gllm02}) for $\alpha\in (0, 2)$:
\begin{gather}\label{eq:mlfuncintegral}
E_{\alpha,\beta}(z)
=\begin{cases}
\frac{1}{\alpha}z^{(1-\beta)/\alpha}e^{z^{1/\alpha}}+\frac{1}{2\pi i\alpha}\int_{\gamma(\epsilon;\delta)}
\frac{e^{\zeta^{1/\alpha}}\zeta^{(1-\beta)/\alpha}}{\zeta-z}d\zeta, & z\in R_+,\\
\frac{1}{2\pi i\alpha}\int_{\gamma(\epsilon;\delta)}
\frac{e^{\zeta^{1/\alpha}}\zeta^{(1-\beta)/\alpha}}{\zeta-z}d\zeta & z\in R_-,
\end{cases}
\end{gather}
where $\gamma(\e; \delta)$ is the curve consisting of $\{re^{-i\delta}: r\ge \e\}$, $\{\e e^{i\theta}: -\delta \le \theta \le \delta\}$ and $\{re^{i\delta}: r\ge \e\}$, going from $\infty e^{-i\delta}$ to $\infty e^{i\delta}$. The parameter $\delta$ satisfies
$\frac{\pi \alpha}{2}<\delta\le \min(\alpha \pi, \pi)$.
The region $R_+$ is the one on the right of $\gamma(\e; \delta)$ while $R_-$ is  on the left.
Note that the two expressions appear different, but they are actually connected continuously across the curve $\gamma(\e; \delta)$.

Next, we recall  the resolvent kernels. In \cite{feng2018continuous}, the resolvent kernels (see \cite{clement1979abstract,clement1981asymptotic,miller1968volterra}) have been used to establish the monotonicity of the solutions to autonomous fractional ODEs. This has then been generalized to discrete schemes on nonuniform meshes in \cite{fengli2023a}.
\begin{definition}\label{def:resol}[see \cite{clement1981asymptotic,miller1968volterra}]
Let $\lambda>0$. The resolvent kernels $r_{\lambda}$ and $s_{\lambda}$ for $a$ are defined respectively by
\begin{gather}
 r_{\lambda}+\lambda r_{\lambda}*a=\lambda a,\quad   s_{\lambda}+\lambda s_{\lambda}* a=1_{t\ge 0}.
\end{gather}
\end{definition}
It has been shown in \cite[Proposition 2.1]{fengli2023a} that both $r_{\lambda}$ and $s_{\lambda}$ are completely monotone for all $\lambda>0$ if $a$ is completely monotone.  For the kernel $a=g_{\alpha}$, $r_{\lambda}$ and $s_{\lambda}$ are thus completely monotone and strictly positive. In fact, it has been mentioned in the proof of \cite[Lemma 3.4]{feng2018continuous} that the resolvent $r_{\lambda}$ for $g_{\alpha}=\frac{1}{\Gamma(\alpha)}t_+^{\alpha-1}$ is
\begin{gather}
r_{\lambda}(t)=-\frac{d}{dt}E_{\alpha}(-\lambda t^{\alpha})>0.
\end{gather}

\subsection{Discretization on nonuniform meshes}\label{subsec:nonunidis}

There are two ways to discretize the fractional ODEs. One is to discretize the differential form \eqref{eq:fracode} even though the derivative may be understood in the generalized one \eqref{eq:capgen}, and the other way is to discretize the integral form \eqref{eq:fracint}.

Let the computational time interval be $[0, T]$, and $0=t_0<t_1<t_2<\cdots<t_N=T$ be the grid points.  Define
\begin{gather}
\tau_n:=t_n-t_{n-1}, \quad n\ge 1.
\end{gather}
We allow $T=\infty$ in some applications and in this case, there are infinitely many grid points $t_n$. Let $u_n$ be the numerical solution at $t_n$.

The integral form \eqref{eq:fracint} may be discretized for $\theta\in [0, 1]$ by
\begin{gather}\label{eq:integraldis}
u_n=u_0+\sum_{j=1}^n \bar{a}_{n-j}^n f_j^{\theta}\tau_j=u_0+\sum_{j=1}^n a_{n-j}^n f_j^{\theta}=:u_0+\cI_{\tau}^{\alpha} f_n^{\theta}.
\end{gather}
Here, $\{\bar{a}_{n-j}^n\}$ is an approximation of the average of $g_{\alpha}(t_n-s)=\frac{1}{\Gamma(\alpha)}(t_n-s)^{\alpha-1}$ on $[t_{j-1}, t_j]$ while $a_{n-j}^n$ is like the integral of $g_{\alpha}(t_n-s)$ on this interval. The notation $f_j^{\theta}$ means an approximation of $f$ at $(1-\theta) t_j+\theta  t_{j-1}$. We consider two examples here. The first is
\begin{gather}\label{eq:ftheta1}
f_j^{\theta}=(1-\theta) f(t_j, u_j)+\theta f(t_{j+1}, u_{j+1})
\end{gather}
while the second is 
\begin{gather}\label{eq:ftheta2}
f_j^{\theta}=f(t_j^{\theta}, u_j^{\theta}), \quad
t_j^{\theta}:=(1-\theta) t_j+\theta  t_{j-1}, \quad u_j^{\theta}:=(1-\theta) u_j+\theta  u_{j-1}.
\end{gather}
These two approximations will have no big difference in our analysis later.

Next, we consider discretization of the differential form \eqref{eq:fracode}.
One can first approximate the derivative and define
\begin{gather}
\nabla_{\tau}u_j:=u_j-u_{j-1}.
\end{gather}
One may then introduce the following approximation inspired by \eqref{eq:captra}
\begin{gather}\label{eq:diffdis}
\cD_{\tau}^{\alpha}u_n:=\sum_{j=1}^n c_{n-j}^n \nabla_{\tau}u_j=f_n^{\theta}, \quad n\ge 1.
\end{gather}
Here, $\nabla_{\tau}u_j\approx u'(t_j)\tau_j$ so $c_{n-j}^n$ is an approximation of the average of $g_{1-\alpha}(t_n-s)$ on $[t_{j-1}, t_j)$.

We remark that \eqref{eq:integraldis} and \eqref{eq:diffdis} can be related by the so-called pseudo-convolution in  \cite{fengli2023a} (also used in \cite{liao2020positive} for defining the ROC kernels).   We arrange the kernel $(a_{n-j}^n)$ into a lower triangular array $A$ of the following form
\begin{gather}\label{eq:arraykernel}
A=\begin{bmatrix}
a_0^{1} &  &  &  &  \\
a_1^{2}& a_0^{2} & & & \\
\cdots & \vdots & \vdots &  &  \\
a_{n-1}^{n} &  \cdots & a_1^{n} & a_0^{n} &\\
\cdots & \vdots & \vdots &   & \vdots\\
\end{bmatrix}.
\end{gather}

The pseudo-convolution $\pc$ between two array kernels is given by the usual matrix product between two arrays of the form \eqref{eq:arraykernel} \cite{fengli2023a}. In particular, $C=A\bar{*}B$ is given by
\begin{gather}\label{eq:convnonuni}
c_{k}^n=\sum_{j=0}^k a_{k-j}^n b_j^{n+j-k}, \quad \text{or}\quad c_{n-k}^n=\sum_{j=k}^n a_{n-j}^n b_{j-k}^j.
\end{gather}
There are some special kernels that play important roles.
The first is $I$ with $I_{n-j}^n=\delta_{nj}$. This is the kernel that is $1$ on the diagonal. The other one is $L$ with $L_{n-j}^n=1$ for all $j\le n$. This kernel corresponds to the Heaviside function $1_{t\ge 0}$ for the continuous case. The inverse of $L$, $L^{(-1)}$, satifies that $(L^{(-1)})_{n-j}^n=1$ for $j=n$, $(L^{(-1)})_{n-j}^n=-1$ for $j=n-1$ and $0$ otherwise. This corresponds to the finite difference operator in \eqref{eq:diffdis}. It can be verified that $I$ is the identify for the pseudo-convolution, and the following properties hold.
\begin{enumerate}[(a)]
\item The pseudo-convolution is associative.
\item If  $B$ is an inverse of $A$, namely $A\pc B=I$, then $B\pc A=I$.
\end{enumerate}
The pseudo-convolution provides us a convenient way to investigate the properties of the kernels. For example, the monotonicity preserving properties of some discretizations has been established using such tools in \cite{feng2023class}.

For a given $A$, the kernel $C_R$ with $A\pc C_R=L$ is called the right complementary kernel. The kernel $C_L$ with $C_L\pc A=L$ is called the left complementary kernel.  The kernel $C_R$ is in fact the so-called RCC kernel (see \cite{liao2023discrete}) and $C_L$ is the  DCC kernel (see \cite{liao2019discrete}).  
If $A$ is a kernel that is invertible, then direct verification tells us that $C_R=A^{(-1)}\pc L$ and $C_L=L\pc A^{(-1)}$.
Moreover, $a_j^n$ is nonincreasing in $n$ if and only if the inverse of $C_R$ has nonpositive off-diagonals; $a_j^n$ is nonincreasing in $j$ if and only if the inverse of $C_L$ has nonpositive off-diagonals.

The pseudo-convolution is also defined between a kernel and a vector $x=(x_1, x_2, \cdots)$:
\begin{gather}
y=A\pc x \quad \Longleftrightarrow  
\quad y_{n}=\sum_{j=1}^{n}a_{n-j}^n x_{j}, \forall n\ge 1,
\end{gather}
then it holds that
$A\pc(B\pc x)=(A\pc B)\pc x$.

One has the following simple conclusion.
\begin{lemma}\label{lmm:schemeequiv}
Consider the kernel $C:=(c_{n-j}^n)$ for $\cD_{\tau}^{\alpha}$ in \eqref{eq:diffdis} and the kernel $A:=(a_{n-j}^n)$ for $\cI_{\tau}^{\alpha}$ in the integral scheme \eqref{eq:integraldis}. If $C$ is the right complementary kernel of $A$, then the two schemes are equivalent.
\end{lemma}

\begin{proof}
It is clear that 
$\cD_{\tau}^{\alpha}u_n=C\pc \nabla_{\tau}u_n=C\pc L^{(-1)}\pc (u-u_0)_n$,
while
$\cI_{\tau}^{\alpha} f(t_n^{\theta}, u_n^{\theta})=A\pc f_n^{\theta}$.
Hence, if $C$ is the right complementary kernel of $A$, $C\pc L^{(-1)}$ is the inverse of $A$ and the claim  follows.
\end{proof}

Below, we will mainly focus on schemes for the differential form. 
We then have
\begin{gather}\label{eq:inverseformula}
B:=A^{-1}=C\pc L^{-1} \Leftrightarrow b_0^n=c_0^n, \quad 
b_{n-j}^n=c_{n-j}^n-c_{n-j-1}^n, j\le n-1.
\end{gather}

Regarding the solvability, the following is straightforward and we omit the proof.
\begin{lemma}
Suppose $f(t,\cdot)$ is uniformly Lipschitz with constant $M$. If $\theta M a_0^n<1$ or equivalently $\theta M <c_0^n$ for all $t_n \le T$, the numerical solution to \eqref{eq:integraldis} or equivalently to \eqref{eq:diffdis} is uniquely solvable. 
\end{lemma}

\begin{remark}
As commented in \cite{feng2023class}, if $c_{0}^n$ is like the average of $g_{1-\alpha}(t_n-s)$, then $a_0^n=1/c_0^n$ is like the integral of $g_{\alpha}$ on $[t_{n-1}, t_n)$ so the discussion above is quite natural. 
\end{remark}

\subsection{Completely positive discretizations}\label{sec:mondis}

We consider a class of variable-step discretizations with the following basic assumptions.
\begin{assumption}\label{ass:cpinverse}
Suppose the array kernel $A$ is invertible and the inverse $B=A^{(-1)}=(b_{n-j}^n)$ satisfies
\begin{gather}\label{eq:Bcondition}
\begin{split}
& b_0^n>0, \quad b_{n-j}^n \le 0, \quad \forall j<n, n\ge 1,\\
& \sum_{j=1}^n b_{n-j}^n \ge 0, \quad \forall n\ge 1.
\end{split}
\end{gather} 
\end{assumption}

For such discretizations, one has the following observation which might be used in applications. The proof uses the signs in $A^{(-1)}$ and is similar to those in \cite[Theorem 3]{li2019discretization} and \cite[Proposition 2.2]{li2021complete}. We thus omit the proof.
\begin{lemma}\label{lmm:convexfunctional}
Suppose $H$ is a Hilbert space with inner product $\langle\cdot, \cdot\rangle$ and $\varphi: H \to \R$ is a convex function. If the kernel $B=C\pc L^{(-1)}$ for $\cD_{\tau}^{\alpha}$ satisfies \eqref{eq:Bcondition}, then
\begin{gather}
\cD_{\tau}^{\alpha}\varphi(u_n)\le \langle \varphi'(u_n), \cD_{\tau}^{\alpha}u_n\rangle.
\end{gather}
In particular, if $u_n\neq 0$, then one has
\[
\cD_{\tau}^{\alpha}\|u_n\| \le \left\langle \frac{u_n}{\|u_n\|}, \cD_{\tau}^{\alpha}u_n \right\rangle.
\]
\end{lemma}

One may consider the resolvent $R_{\lambda}$ defined using the pseudo-convolution similar to the continuous case:
\begin{gather}\label{eq:disreldef}
R_{\lambda}+\lambda R_{\lambda}\pc A=\lambda A \quad
\Longleftrightarrow \quad
A-R_{\lambda}\pc A=\frac{1}{\lambda}R_{\lambda}.
\end{gather}
It has been proved in \cite[Theorem 5.1]{fengli2023a} that
\begin{proposition}
The array kernel $A$ satisfies Assumption \ref{ass:cpinverse} if and only if for all $\lambda>0$, $R_{\lambda}$ has nonnegative entries, $0<(R_{\lambda})_0^n<1$
and $\sum_{j=1}^n (R_{\lambda})_{n-j}^n\le 1$.
\end{proposition}
This is the discrete analogue of the complete positivity of the continuous kernels studied in \cite{clement1981asymptotic}, namely those with nonnegative resolvents. Hence, we say a discretization satisfying Assumption \ref{ass:cpinverse} is {\it completely positive}. The following may be used sometimes.
\begin{lemma}
If $a_{j-1}^{n-1}\ge a_j^n$ for all $n\ge 2$ and $j\le n-1$, 
and the inverse $B$ satisfies 
\[
b_0^n>0, \quad b_{n-j}^n \le 0, \quad \forall j<n, n\ge 1,
\]
then $A$ is completely positive.
\end{lemma}
In fact, if $b_0^n>0$ and $b_{n-j}^n\le 0$ for all $j<n$, $A$ must have nonnegative entries by \cite[Lemma 4.3]{fengli2023a}. Then, together with the fact that
 $a_{j-1}^{n-1}\ge a_j^n$, one can easily verify that $C=A^{(-1)}\pc L$ is nonnegative.

The resolvent kernels have the following simple facts as proved in \cite{fengli2023a}. 
\begin{lemma}\label{lmm:resolventprop}
Suppose the diagonal elements of $A$ are positive.  The resolvent kernel $R_{\lambda}$  always exists for $\lambda>0$ and  $R_{\lambda}\pc A=A\pc R_{\lambda}$.
Moreover, as $\lambda\to \infty$:
\[
R_{\lambda}=I-\lambda^{-1}A^{(-1)}+O(\lambda^{-2}),
\]
where the $O(\lambda^{-2})$ is elementwise.
\end{lemma}

Below we look at examples of discretizations on nonuniform meshes. We say a discretization is completely positive if the corresponding kernel $A$ is completely positive (note that the schemes for differential form and the integral form are equivalent by Lemma \ref{lmm:schemeequiv}).
 The L1 scheme \cite{lin2007finite,stynes2017error} is the most popular and simplest discretization, obtained by piecewise linear interpolation of the derivative $u$ in the differential form, which can be written as
\begin{gather}\label{eq:L1a}
D^{\alpha}u(t_n) \approx \cD_{\tau}^{\alpha}u_n:=C\pc \nabla_{\tau} u_n=C\pc L^{(-1)}\pc (u-u_0)_n,
\end{gather}
where
\begin{gather}\label{eq:L1b}
c_{n-j}^n=\frac{1}{\tau_j\Gamma(1-\alpha)}\int_{t_{j-1}}^{t_j}(t_n-s)^{-\alpha}\,ds.
\end{gather}
It can be verified directly that $B^{(-1)}:=C\pc L^{(-1)}$ satisfies \eqref{eq:Bcondition} so that $A$ is completely positive.
Hence, the  $L1$ scheme is completely positive.

One can also approximate the integral formulation. 
In \cite[Theorem 6.2]{feng2023class}, the following integral scheme
\begin{gather}\label{eq:fodeintegralscheme}
u_n=u_0+\sum_{j=1}^n a_{n-j}^n f(t_j, u_j), \quad a_{n-j}^n=\frac{1}{\Gamma(\alpha)}\int_{t_{j-1}}^{t_j}(t_n-s)^{-\alpha}\,ds
\end{gather}
is proposed as approximation and has been proved to be completely positive.

For uniform schemes, all the CM-preserving schemes discussed in \cite{li2021complete} are completely positive. This includes the standard Gr\"unwald-Letnikov (GL) scheme, the convolutional quadrature (CQ) with $\theta$-method etc. The GL and CQ methods are not easy to generalize to nonuniform meshes. The requirement on completely positive schemes is quite reasonable. Some related discussions can be found in \cite{liao2019discrete} and \cite{liao2023discrete}. The schemes considered in \cite{liao2019discrete} are a subclass of completely positive schemes.  Moreover, the second order Crank-Nicolson scheme with the $L1^+$ discretization can also be changed into a modified completely positive scheme, as we shall see in section \ref{subsec:numericalcp}.
\begin{remark}
The completely positive schemes may have some order barrier, which means that the high order schemes cannot be completely positive. Often, in these high order schemes, only the first few $b_j^n$ terms (like $j=1,2$) cause issues. It is possible to decompose the kernel for a high order scheme into a completely positive kernel and the kernel for some local difference operator. We will leave this for future study.
\end{remark}

\section{The comparison principles}\label{sec:comparisonprinciple}
In this section, we establish some comparison principles. First, we give an alternative proof for continuous problems based on the resolvents. Then, we generalize the proof to discrete schemes.

\subsection{Comparison principles for time continuous case}

\begin{theorem}\label{thm:contcomp}
Suppose $f$ is continuous and locally Lipschitz in $u$. Let $y$ and $z$ be two continuous functions satisfying
\[
D_c^{\alpha}y\ge f(t, y), \quad D_c^{\alpha}z\le f(t, z)
\]
in the distributional sense, with $y(0)\ge z(0)$, then $y(t)\ge z(t)$ on the common interval of existence. If $y(0)>z(0)$, then $y(t)>z(t)$ on the interval considered. 
\end{theorem}

This result is in fact known. One can refer to \cite[Theorem 2.3]{ramirez2012generalized}, or \cite[Theorem 2.2]{fllx18note} for the generalized version with less regularity.  However, the proof there relies on the continuity of the continuous functions, which is not applicable for discrete sequences. We provide another proof using the resolvents, which can then be generalized the discrete schemes.

By the definition \eqref{eq:capgen},  $D_c^{\alpha}u=g_{-\alpha}*(\theta(t)(u-u_0))$.
Using the definition of the resolvent, we find that  $r_{\lambda}*g_{-\alpha}=\lambda(\delta-r_{\lambda})$. Hence,  one has 
\begin{gather}\label{eq:resolcap}
\lambda^{-1}r_{\lambda}*D_c^{\alpha}u=\theta(t)(u-u_0)-r_{\lambda}*(u-u_0).
\end{gather}
If $u$ is absolutely continuous, one can see this more clearly. In fact, $g_{-\alpha}*(\theta(t)(u-u_0))
=g_{1-\alpha}*(\theta u')$. Here, $g_{1-\alpha}$ is the complementary kernel of $g_{\alpha}$
in the sense that $g_{1-\alpha}*g_{\alpha}=\theta(t)=1_{t\ge 0}$. It can then be derived from the definition of the resolvent that
\begin{gather}
r_{\lambda}*g_{1-\alpha}=\lambda s_{\lambda}=\lambda (1_{t\ge 0}-r_{\lambda}*1_{t\ge 0}).
\end{gather}

The main observation is that $r_{\lambda}$ is nonnegative so that the inequality can be preserved if we convolving both sides with $\lambda^{-1}r_{\lambda}$. 

\begin{proof}[Proof of Theorem \ref{thm:contcomp}]
Define $u:=y-z$. Taking the difference between the two relations, one has (in the distributional sense) that
\[
D_c^{\alpha}u\ge h(t) u(t), \quad
h(t)=\int_0^1 \partial_uf(\eta y+(1-\eta)z)d\eta.
\]
Convolving both sides with $\lambda^{-1}r_{\lambda}$, which is nonnegative, one then has
\[
(u-u_0)-r_{\lambda}*(u-u_0)\ge \lambda^{-1}r_{\lambda}*(hu).
\]
This implies that
\begin{gather}\label{eq:compineq}
u(t)\ge (u_0-r_{\lambda}*u_0)+r_{\lambda}*(u+\lambda^{-1} h u)
=u_0s_{\lambda}+\int_0^t r_{\lambda}(t-s)(1+h(s)/\lambda)u(s)\,ds.
\end{gather}

If $u_0>0$, one can see that $u(t)>0$ for $t$ small enough. Then, 
for all $t>0$ considered, we can take $\lambda$ large enough such that $1+h(s)/\lambda>0$ on $[0, t]$, and \eqref{eq:compineq} implies that $u(t)>0$.

Consider $u_0=0$. If $u(t)<0$ somewhere, we set $t_1\ge 0$ to be the first time when $u(t)<0$ on $(t_1, t_1+\e)$ for some $\e>0$. Then, $u(t_1)=0$. It is clear that $u(t)=0$ on $[0, t_1]$, otherwise \eqref{eq:compineq} would yield a contradiction by setting $t=t_1$ and $\lambda$ large enough. Take $\lambda$ large enough such that $1+h(s)/\lambda>0$ on $(t_1, t_1+\e)$ and set
\[
A:=\sup_{s\in (t_1, t_1+\e)} 1+h(s)/\lambda.
\]
Take $\delta\le \e$ such that $\int_0^{\delta}r_{\lambda}(s)\,ds\le 1/(2A)$.
Moreover, let
\[
t_2=\mathrm{argmin}_{s\in [t_1, t_1+\delta]}u(s)\in (t_1, t_1+\delta].
\]
Then, $u(t_2)<0$ and
\[
u(t_2)\ge \int_{t_1}^{t_2} r_{\lambda}(t_2-s)(1+h(s)/\lambda)u(s)\,ds
\ge Au(t_2)\int_{t_1}^{t_2} r_{\lambda}(t_2-s)\,ds
\ge u(t_2)/2.
\]
This is then a contradiction. Hence, the claim is proved.
\end{proof}

If the inequality is given in integral form, the technique here does not apply, since $\delta-r_{\lambda}$ is not nonnegative. 
 Currently, the comparison principle is known to hold for integral forms only when the function $f(t, \cdot)$ is nondecreasing as in \cite{liliu18frac}. Whether it can be generalized to general $f$  is an interesting question.

\subsection{Comparison principles for numerical schemes}\label{subsec:numericalcp}

The argument in \cite[Theorem 2.3]{ramirez2012generalized} or \cite[Theorem 2.2]{fllx18note} cannot be generalized to discrete schemes easily as it relies on the continuity. Motivated by the proof above based on resolvents and the pseudo-convolution, we are then able to establish a series of comparison principles for the discrete schemes.

The following is for the implicit scheme ($\theta=1$).
\begin{theorem}\label{thm:implicitcp}
Suppose that the discretization is completely positive. If $y_0\ge z_0$ and
\[
\cD_{\tau}^{\alpha} y_n\ge f(t, y_n),
\quad \cD_{\tau}^{\alpha} z_n\le f(t, z_n), \forall n\ge 1,
\]
then the following comparison principle holds. 
\begin{enumerate}[(1)]
\item If $f(t,\cdot)$ is nonincreasing for all $t$, then $y_n \ge z_n$ for all $n$.  If $C$ has positive entries, then $y_0>z_0$ implies that $y_n>z_n$.
\item Suppose that $f(t,\cdot)$ is uniformly Lipschitz such that $|\partial_u f(t, u)|\le M$ for all $t,u$, 
and the discretization satisfies the stability condition $M/c_0^n<1$, then one has $y_n\ge z_n$.  If $C$ has positive entries, then $y_0>z_0$ implies that $y_n>z_n$.
\end{enumerate}
\end{theorem}

\begin{proof}
Define $u_n=y_n-z_n$. Then, it holds that
\begin{gather}\label{eq:cpaux1}
\cD_{\tau}^{\alpha}u_n \ge h_n u_n,
\end{gather}
where $h_n=\int_0^1 \partial_uf(\eta y_n+(1-\eta)z_n)\,d\eta$.
Recall
\[
\cD_{\tau}^{\alpha}u_n=C\pc(\nabla_{\tau}u_n)=A^{(-1)}\pc u_n.
\]
Taking pseudo-convolution with $\lambda^{-1}R_{\lambda}$ on both sides of \eqref{eq:cpaux1} and noting that $R_{\lambda}$ has nonnegative entries, one has
\begin{gather}\label{eq:auxdiscreteresol}
u_n\ge u_0\Big(1-\sum_{j=1}^n (R_{\lambda})_{n-j}^n \Big)
+R_{\lambda}\pc [(1+h_n/\lambda )u_n].
\end{gather}
Hence,
\[
(1-(R_{\lambda})_0^n(1+h_n/\lambda))u_n \ge 
u_0\left(1-\sum_{j=1}^n (R_{\lambda})_{n-j}^n\right)+\sum_{j=1}^{n-1}(R_{\lambda})_{n-j}^n (1+h_j/\lambda )u_j.
\]
 Assume that $u_j\ge 0$ for $j\le n-1$ has been established.
For fixed $n$, one can choose $\lambda$ large enough such that the coefficients on the right hand side are all nonnegative, and also such that $1-\sum_{j=1}^n (R_{\lambda})_{n-j}^n>0$ if $C$ has positive entries by Lemma \ref{lmm:resolventprop}.

If $f$ is nonincreasing, $1-(R_{\lambda})_0^n-h_n/\lambda>0$ always holds, then simple induction yields that
$u_n\ge 0$ if $u_0\ge 0$.

If $f$ is assumed to be Lipschitz, then by Lemma \ref{lmm:resolventprop}, $(R_{\lambda})_0^n=1-\lambda^{-1}b_0^n+O(\lambda^{-2})$ and $b_0^n=c_0^n$, then
\[
1-(R_{\lambda})_0^n(1+h_n/\lambda)=\lambda^{-1}(c_0^n-h_n)+O(\lambda^{-2}).
\]
Hence, if $c_0^n>M$, one can choose $\lambda$ large enough such that the coefficient is positive. Then, $u_n>0$.
\end{proof}

Taking $f(t, u)\equiv 0$, then $\cD_{\tau}^{\alpha}y_n \ge 0$ implies that $y_n \ge 0$, which gives the following.
\begin{corollary}
For any completely positive discretization, if $\cD_{\tau}^{\alpha}x_n
\ge \cD_{\tau}^{\alpha}y_n$ and $x_0\ge y_0$, then $x_n \ge y_n$.
\end{corollary}

Below, we consider weighted implicit schemes  for $\theta\in [0, 1)$. 
\begin{theorem}\label{thm:cmpweighted}
Assume the discretization is completely positive. Suppose $y$ and $z$ are two sequences with $y_0\ge z_0$ and satisfy the following for $\theta\in [0, 1)$:
\begin{gather}
\cD_{\tau}^{\alpha}y_n\ge  f_n^{\theta}[y],\quad  \cD_{\tau}^{\alpha}z_n \le  f_n^{\theta}[z],
\end{gather}
where $f_n^{\theta}[y]$ is either given by \eqref{eq:ftheta1} or \eqref{eq:ftheta2} (with $u_j$ replaced by $y_j$), and $f_n^{\theta}[z]$ is similarly defined.
\begin{enumerate}[(1)]
\item Suppose $f$ is nondecreasing in $u$. If $c_0^n>\theta M$, then $y_n \ge z_n$.

\item If $f(t,\cdot)$ is uniformly Lipschitz with constant $M$, $c_0^n >\theta M$ and $c_0^n-c_1^n\ge (1-\theta)M$, then $y_n \ge z_n$.
Moreover, if $C$ has positive entries,  then $y_0>z_0$ implies that $y_n>z_n$.
\end{enumerate}

\end{theorem}

\begin{proof}
Define $u_n=y_n-z_n$. It holds that
 \[
 \cD_{\tau}^{\alpha}u_n=C_R\pc L^{(-1)}\pc (u-u_0)_n \ge \theta h_{n,1} u_n+(1-\theta)h_{n,2} u_{n-1},
 \]
 In the case of \eqref{eq:ftheta1},
$ h_{n+1,2}=h_{n,1}=\int_0^1 \partial_uf(t_n, \eta y_n+(1-\eta)z_n)\,d\eta$.
 In the case of \eqref{eq:ftheta2},
$h_{n,1}=h_{n,2}=\int_0^1 \partial_uf(t_j^{\theta}, \eta y_n^{\theta}+(1-\eta)z_n^{\theta})\,d\eta$.
 
Taking the pseudo-convolution with $R_{\lambda}$ on left for both sides, one has by the nonnegativity of the elements of $\lambda^{-1}R_{\lambda}$ that
\[
u_n-u_0-R_{\lambda}\pc(u-u_0) \ge \lambda^{-1}R_{\lambda}\pc (\theta h_{n,1}u_n+(1-\theta) h_{n,2}u_{n-1}),
\]
which implies that
\begin{multline}\label{eq:aux1}
(1-(R_{\lambda})_0^n(1+\theta h_{n,1}/\lambda))u_n
\ge \left(1-\sum_{j=1}^n (R_{\lambda})_{n-j}^n+\lambda^{-1}(1-\theta)(R_{\lambda})_{n-1}^n h_{1,2}\right)u_0 \\
+\sum_{j=1}^{n-1}\left((R_{\lambda})_{n-j}^n(1+\theta h_{j,1}/\lambda)+\lambda^{-1}(1-\theta)(R_{\lambda})_{n-j-1}^n h_{j+1,2})\right)u_j.
\end{multline}
The observation is that $u_n$ does not depend on $\lambda$ so for each $n$ one can inductively take $\lambda$ large enough to show $u_n \ge 0$.

If $f(t, \cdot)$ is nondecreasing, then $h_{n,\ell} \ge 0$, $\ell=1,2$ and the coefficients on the right hand side of  \eqref{eq:aux1} are all nonnegative.
By Lemma \ref{lmm:resolventprop} and \eqref{eq:inverseformula}, 
\[
(R_{\lambda})_0^n =1-\lambda^{-1}c_0^n+O(\lambda^{-2}).
\]
Since $|h_{n,\ell}|\le M$, one can see that if $c_0^n >\theta M$, then 
\[
1-(R_{\lambda})_0^n(1+\theta h_{n,1}/\lambda)>0
\]
for $\lambda$ large enough.
This then implies that $u_n \ge 0$ for all $n$ if $u_0\ge 0$.

Now, we focus on the second case.  Due to the same reason as above,
if $c_0^n >\theta M$, then the coefficient of $u_n$ is positive when $\lambda$ is large enough.  
Again, by Lemma \ref{lmm:resolventprop} and \eqref{eq:inverseformula}, one has
\[
\sum_{j=1}^n (R_{\lambda})_{n-j}^n=1-\lambda^{-1}\sum_{j=1}^n b_{n-j}^n
+O(\lambda^{-1})=1-\lambda^{-1}c_{n-1}^n+O(\lambda^{-2}).
\]
Moreover, $(R_{\lambda})_{n-1}^n=-\lambda^{-1}b_{n-1}^n+O(\lambda^{-2})$. Hence, the leading term in the coefficient of $u_0$
is $\lambda^{-1}c_{n-1}^n$ with $c_{n-1}^n\ge 0$.  For the last term in \eqref{eq:aux1}, when $j=n-1$, the coefficient is 
\[
(R_{\lambda})_1^n(1+\theta h_{n-1,1}/\lambda)+\lambda^{-1}(1-\theta)(R_{\lambda})_{0}^n h_{n,2}
=\lambda^{-1}(-b_1^n+(1-\theta)h_{n,2})+O(\lambda^{-2}),
\quad \lambda\to\infty.
\]
If $c_0^n-c_1^n\ge  (1-\theta)M$, then noting $-b_1^n=c_0^n-c_1^n$ and $|h_{n-1}|\le M$, one finds that
the leading term in the coefficient for $j=n-1$ is $\lambda^{-1}(-b_1^n+(1-\theta)h_{n,2})$ with $-b_1^n+(1-\theta)h_{n,2}\ge 0$.
For $j<n-1$, the coefficient is like $(R_{\lambda})_{n-j}^n+O(\lambda^{-2})=\lambda^{-1}(-b_{n-j}^n)+O(\lambda^{-2})$. Since $-b_{n-j}^n=c_{n-j-1}^n-c_{n-j}^n\ge 0$. Then, one can multiply $\lambda$ on both sides to take the limit $\lambda\to\infty$. This then yields that $u_n\ge 0$ by simple induction. 

If $C$ has positive entries,  $c_{n-1}^n>0$. The coefficient of $u_0$ is like $\lambda^{-1}c_{n-1}^n$ when $\lambda$ is large enough. The coefficients in other terms at this order are all nonnegative. This will then naturally yield $u_n>0$ by choosing $\lambda$ sufficiently large.
 \end{proof}

For a typical scheme, since $c_j^n$ is like the average of $g_{1-\alpha}(t_n-s)$ on $(t_{n-j-1}, t_{n-j})$, 
\[
c_0^n\sim \frac{1}{\Gamma(2-\alpha)}\tau_n^{-\alpha},
\quad c_1^n\sim \frac{1}{\Gamma(2-\alpha)}\frac{(\tau_n+\tau_{n-1})^{1-\alpha}-\tau_n^{1-\alpha}}{\tau_{n-1}}.
\]
Hence, when the step sizes are small, the conditions listed above are expected to hold.

Below, we perform a discussion for the Crank-Nicolson scheme.
Consider the so-called $L1^+$ scheme for the derivative at $t_{n-1/2}$:
\begin{gather}
(\cD_{\tau}^{\alpha}u)^{n-1/2}=\frac{1}{\tau_n}\int_{t_{n-1}}^{t_n}
\int_0^t g_{1-\alpha}(t-s)(\Pi_1 u)'(s)\,ds,
\end{gather}
where $\Pi_1 u(s)$ is the piecewise linear approximation of $u$
such that $\Pi_1 u(t_j)=u_j$. Define
\[
\chi_{n-j}^n=\frac{1}{\tau_n \tau_k}\int_{t_{n-1}}^{t_n}
\int_{t_{k-1}}^{\min(t, t_k)}g_{1-\alpha}(t-s)\,ds dt,\quad 1\le j\le n.
\]
Then, the Crank-Nicolson scheme is given by
\begin{gather}
(\cD_{\tau}^{\alpha}u)^{n-1/2}=\sum_{j=1}^n\chi_{n-j}^n\nabla_{\tau}u_j=f_n^{1/2}[u].
\end{gather}

It has been shown in \cite{liao2023discrete} that $\chi_j^n$ is not monotone in $j$ because $\chi_0^n$ could be smaller than $\chi_1^n$.
Only when $\alpha>\alpha_c$ for some critical value $\alpha_c\in (0, 1)$, one has $\chi_0^n>\chi_1^n$. If one defines
\begin{gather}
c_0^n=2\chi_0^n,\quad c_j^n=\chi_j^n, j\ge 1,
\end{gather}
then $C=(c_{n-j}^n)$ is doubly monotone and it satisfies the log-convex condition $c_{j-1}^{n-1}c_{j+1}^n \ge c_j^n c_j^{n-1}$.  The kernel associated with $C$ is completely positive. The Crank-Nicolson scheme can then be written as
 \begin{gather}
 C\pc \nabla_{\tau}u_n= C\pc L^{(-1)}\pc (u-u_0)_n=\chi_0^n(u_n-u_{n-1})
 +f_n^{1/2}[u].
 \end{gather}

\begin{proposition}
If $\alpha>\alpha_c$, the Crank-Nicolson scheme satisfies the following comparison principle. Suppose $f(t,\cdot)$ is uniformly Lipschitz with constant $M$, and two sequences $\{y_n\}$, $\{z_n\}$ satisfy
\[
(\cD_{\tau}^{\alpha}y)^{n-1/2}\ge f_n^{1/2}[y],
\quad (\cD_{\tau}^{\alpha}z)^{n-1/2}\le f_n^{1/2}[z],
\]
with $y_0\ge z_0$. If the discretization satisfies that $\chi_0^n>M/2$ and $\chi_0^n-\chi_1^n\ge M/2$, then
\[
y_n \ge z_n.
\]
If the function is nondecreasing, one only needs $\chi_0^n>M/2$ for the comparison principle to hold.
\end{proposition}

The argument is almost the same as before.  The sequence $u_n=y_n-z_n$ satisfies 
\[
C\pc L^{(-1)}\pc(u_n-u_0)\ge \chi_0^n(u_n-u_{n-1})+\frac{1}{2}(h_{n,1}u_n+h_{n,2}u_{n-1}),
\]
where $h_{n,\ell}$ ($\ell=1,2$) are the same as in the proof of Theorem \ref{thm:cmpweighted}.  
Note that $C$ corresponds to a completely positive kernel as mentioned, the argument above can then be carried here with minor modification. We thus skip the proof here. 

We remark that some versions of the comparison principles above can also be established using the technique as in the proof of \cite[Theorem 3]{li2019discretization}, using the signs of $A^{(-1)}$. However, such a proof heavily relies on the properties in the discretized scheme and has no analogue for the time continuous version.

\section{Gr\"onwall inequalities for completely positive schemes}\label{sec:gronwall}

In this section, we will establish some Gr\"onwall inequalities. The versions for $f(u)=-\lambda u+c$, $\lambda>0$ could be used for uniform error control and decay estimates for dissipative systems. We will only consider implicit schemes. The Gr\"onwall inequality for the weighted implicit or explicit schemes can be obtained similarly.

We start with some basic facts under the following assumption.
\begin{assumption}\label{ass:gron1}
If there exists $\nu>0$ such that for all $n$ with $t_n \le T$, one has
\begin{gather}\label{eq:clowerbound}
c_{n-j}^n \ge \nu \frac{1}{\tau_j}\int_{t_{j-1}}^{t_j} g_{1-\alpha}(t_n-s)\,ds,\quad \forall 1\le j\le n.
\end{gather}
\end{assumption}
This assumption says that $c_{n-j}^n$ is bounded by from below by a fraction of the average of $g_{1-\alpha}(t_n-\cdot)$ on $(t_{j-1}, t_j]$, which is clearly natural for the approximation of the continuous derivative in \eqref{eq:captra}.

\begin{lemma}\label{lmm:concavecomp}
Under assumption \ref{ass:gron1}, one has for concave function $v(\cdot)$ with $v'(\cdot)\ge 0$ that
\begin{gather}
 \cD_{\tau}^{\alpha} v(t_n) 
 \ge \nu  \frac{1}{\Gamma(1-\alpha)}\int_0^{t_n}(t_n-s)^{-\alpha} v'(s)\,ds
 =\nu D^{\alpha}v(t_n),\quad t_n\le T.
\end{gather}
Moreover, for the $L1$ scheme, $\nu=1$.
\end{lemma}

\begin{proof}
By assumption \ref{ass:gron1}, consider each term in $\cD_{\tau}^{\alpha}v(t_n)$, one has
\[
c_{n-j}^n (v(t_j)-v(t_{j-1}))
\ge \nu \frac{1}{\tau_j}\int_{t_{j-1}}^{t_j}g_{1-\alpha}(t_n-s)\,ds
\int_{t_{j-1}}^{t_j}v'(s)\,ds.
\]
Here, $v'(s)$ is nonincreasing. By Chebyshev's sorting inequality \cite[item 236]{hardy1952inequalities}, one has
\[
\frac{1}{\tau_j}\int_{t_{j-1}}^{t_j}g_{1-\alpha}(t_n-s)\,ds
\int_{t_{j-1}}^{t_j}v'(s)\,ds\ge  \int_{t_{j-1}}^{t_j}g_{1-\alpha}(t_n-s) v'(s)\,ds.
\]
In fact, a direct computation can also verify this:
\begin{multline*}
\int_{t_{j-1}}^{t_j}\int_{t_{j-1}}^{t_j} g_{1-\alpha}(t_n-s)
v'(z)\,dsdz
-\int_{t_{j-1}}^{t_j}\int_{t_{j-1}}^{t_j}g_{1-\alpha}(t_n-s) v'(s)\,ds dz\\
=\int_{t_{j-1}}^{t_j}\int_{t_{j-1}}^{s} [g_{1-\alpha}(t_n-s)-g_{1-\alpha}(t_n-z)]
(v'(z)-v'(s))\,dsdz \ge 0.
\end{multline*}
The last equality above is obtained by changing the order of integration for the second term. 
Summing the estimate above over $j$ and using \eqref{eq:clowerbound}, one then obtains the desired result.
\end{proof}

As a special case, 
\begin{gather}\label{eq:derilower}
 \cD_{\tau}^{\alpha} \left(\frac{1}{\Gamma(1+\alpha)}t_n^{\alpha}\right) \ge  \nu,
 \quad t_n\le T.
\end{gather}

\subsection{A Gr\"onwall inequality for uniform bound}

Next, we will consider $T=\infty$ and suppose Assumption \ref{ass:gron1} holds for all $t_n<\infty$. We  aim to establish a Gr\"onwall inequality that is useful for the uniform bound if the system is dissipative. This could be useful for uniform-in-time error estimate.

\begin{theorem}\label{thm:unibound}
Suppose Assumption \ref{ass:gron1} holds for all $n>0$. Consider an arbitrary completely positive scheme. Suppose a nonnegative sequence satisfies for some $\lambda\ge 0$ that
\[
\cD_{\tau}^{\alpha}v_n \le -\lambda v_n+c.
\]
Then, for $\lambda>0$ and $v_0\le c/\lambda$, one has
\begin{gather}
v_n \le (v_0-c/\lambda)E_{\alpha}(-\nu^{-1}\lambda t_n^{\alpha})+c/\lambda.
\end{gather}
If $\lambda=0$ and $v_0\ge 0$, one has
\[
v_n \le v_0+\nu^{-1}c\frac{1}{\Gamma(1+\alpha)}t_n^{\alpha}.
\]
\end{theorem}

\begin{proof}

Consider only $\lambda>0$ ($\lambda=0$ is similar). Consider the auxiliary problem
\[
D^{\alpha}y^{\nu}=\nu^{-1}(-\lambda y^{\nu}+c), y^{\nu}(0)=v_0.
\]
If $v_0 \le c/\lambda$,  the solution is concave using the explicit formula 
\eqref{eq:exactsol} and the fact that $t\mapsto E_{\alpha}(-\lambda t^{\alpha})$ is completely monotone.  Then,
\[
\cD_{\tau}^{\alpha}y^{\nu}\ge \nu(\nu^{-1}(-\lambda y^{\nu}+c))
=-\lambda y^{\nu}+c.
\]
Since $v\mapsto -\lambda v+c$ is nonincreasing, the comparison principle in Theorem \ref{thm:implicitcp} holds for any completely positive discretization. Then, it holds that $v_n\le y^{\nu}(t_n)$, which is the desired result.
\end{proof}

This results above hold with no restriction on the step size ratio and the largest step size. It implies that the solution of the implicit scheme $\cD_{\tau}^{\alpha}v_n=-\lambda v_n+c$ for $L1$ scheme is below the exact solution for such case. In fact, if $\alpha=1$ (the ODE case), this is known $v_n=(v_0-c/\lambda)(1+\lambda\tau)^{-n}+c/\lambda \le (v_0-c/\lambda)e^{-\lambda n\tau}+c/\lambda=y(t_n)$.

Often for a dissipative system, one may obtain that the error satisfies
\[
\cD_{\tau}^{\alpha}\|e_n\|^2\le -\lambda \|e_n\|^2+C\tau^{\beta},
\]
for some $\beta>0$. Then, provided that $\|e_0\|=0$, the error is controlled uniformly in time as
\[
\|e_n\|^2\le \frac{C\tau^\beta}{\lambda}(1-E_{\alpha}(-\nu^{-1}\lambda t_n^{\alpha})) \le \frac{C}{\lambda}\tau^{\beta}.
\]

\subsection{A Gr\"onwall inequality for decay estimates}

In this subsection, we consider a Gr\"onwall inequality that can control a sequence above by a decreasing sequence. In particular, consider the discrete inequality
\[
\cD_{\tau}^{\alpha}u_n \le f(u_n),
\]
where $f(u_0)<0$.  We aim to find an upper boud for $u_n$.

We first explain our strategy. Suppose that for a class of nonnegative functions $w(t)$, there is a constant $\rho$ such that
\begin{gather}\label{eq:decayaux1}
c_{n-j}^n\int_{t_{j-1}}^{t_j}w(s)\,ds\le \rho \int_{t_{j-1}}^{t_j}g_{1-\alpha}(t_n-s)w(s)\,ds, \quad
\forall j\le n.
\end{gather}
Then, consider the function $u^{\rho}(\cdot)$ solving
\[
D^{\alpha}u^{\rho}=\rho^{-1}f(u^{\rho}),\quad u^{\rho}(0)= u_0.
\]
Then $t\mapsto u^{\rho}(t)$ is nonincreasing and $w(t)=-\frac{d}{dt}u^{\rho}(t)\ge 0$ for $t>0$. 
If $w(\cdot)$ is in the class such that \eqref{eq:decayaux1} holds, one then has 
\[
\cD_{\tau}^{\alpha}(u^{\rho}(t_n))\ge \rho D^{\alpha}u^{\rho}=f(u^{\rho}).
\]
The comparison principle in Theorem \ref{thm:implicitcp} implies that
$u_n \le u^{\rho}(t_n)$.

To establish \eqref{eq:decayaux1}, we introduce another assumption.
\begin{assumption}\label{ass:cupper}
There exists $\rho_1>0$ such that
\begin{gather}\label{eq:compkernealub}
c_{n-j}^n \le \rho_1 \frac{1}{\tau_j}\int_{t_{j-1}}^{t_j} g_{1-\alpha}(t_n-s)\,ds.
\end{gather}
\end{assumption}
This assumption is a mirror version of Assumption \ref{ass:gron1}, which says that the $c_{n-j}^n$ is bounded above by a multiple of the average of $g_{1-\alpha}(t_n-\cdot)$ on $(t_{j-1}, t_j]$. We emphasize that these two assumptions put no restriction on the step size ratio. In fact, for L1 schemes,  $c_{n-j}^n$ equals exactly to the average so these two assumptions hold. Clearly, for any nonuniform grid, one can define the L1 scheme .

As a consequence of the following simple observation
\begin{gather*}
\frac{1}{\tau_j}\int_{t_{j-1}}^{t_j}(t_n-s)^{-\alpha}\,ds
\le \frac{1}{t_n-t_{j-1}}\int_{t_{j-1}}^{t_n}(t_n-s)^{-\alpha}\,ds  \le \frac{1}{1-\alpha} (t_n-s')^{-\alpha}, \quad \forall s'\in (t_{j-1}, t_j),
\end{gather*}
one has the following:
\begin{lemma}\label{lmm:firstcontrolrho1}
Suppose Assumption \ref{ass:cupper} holds. Then relation \eqref{eq:decayaux1} holds with 
$\rho=\frac{\rho_1}{1-\alpha}$.
\end{lemma}

The above constant $\rho$ is not satisfactory. In fact, if $\alpha$ is close to $1$, it blows up, which should not happen in practice.  Below, we will consider the special case
\[
f(u)=-\lambda u+c, \quad \lambda>0
\]
and make use of the information of the solution given by \eqref{eq:exactsol}, with $\beta=-\lambda$. Clearly, it suffices to consider 
$y(t)=E_{\alpha}(-\lambda t^{\alpha})$.
 Define 
\begin{gather}\label{eq:w4}
w(t):=-y'(t).
\end{gather}  
Since $y(\cdot)$ is completely monotone, $w$ is also completely monotone.
Our observation is that there could be another constant $\sigma$ such that
\begin{gather}\label{eq:wsigma}
\frac{1}{\tau_j}\int_{t_{j-1}}^{t_j}w(s)\,ds
\le \sigma w(t_j),
\end{gather}
which is good for $\alpha\in [1/2, 1)$. Then, one can take $\rho=\rho_1\min(1/(1-\alpha), \sigma)$ for $\alpha\in [1/2, 1)$.

\begin{lemma}\label{lmm:upperbound}
The function $w(t)$ in \eqref{eq:w4} is a completely monotone function so that
$t\mapsto \frac{w(t+\tau)}{w(t)}$ is nondecreasing for any $\tau>0$.
Consequently, for $\alpha\in [1/2, 1)$ and $\tau_n$ with $\lambda\tau_n^{\alpha}\le 1$, one has for two universal constants $c_1, c_2$ such that
\begin{gather}
\frac{1}{\tau_j}\int_{t_{j-1}}^{t_j}w(s)\,ds \le \frac{c_2}{c_1}\max(2^{1-\alpha}, \alpha^{-1}) w(t_j).
\end{gather}
One can thus take $\sigma=2c_2/c_1>1$ for \eqref{eq:wsigma}.
\end{lemma}
\begin{proof}
Since $y(\cdot)$ is completely monotone, 
$w(t)=-y'(t)$ is completely monotone by definition.
By the discussion in \cite{clement1979abstract,miller1968volterra}, $t\mapsto w(t+\tau)/w(t)$ is nondecreasing. Let $n_0$ be the smallest one such that $\lambda t_{n_0}^{\alpha}\ge 1$. For all $t\ge t_{n_0}$ and $\tau$ with $\lambda \tau^{\alpha}\le 1$, one has $\tau\le t_{n_0}$. Recalling that 
$w(t)=\alpha \lambda t^{\alpha-1}(\alpha^{-1}E_{\alpha, \alpha}(-\lambda t^{\alpha}))$, one  thus has
\[
\frac{w(t)}{w(t+\tau)}\le \frac{w(t_{n_0})}{w(t_{n_0}+\tau)}
=\left(\frac{t_{n_0}}{t_{n_0}+\tau}\right)^{\alpha-1}\frac{\alpha^{-1}E_{\alpha,\alpha}(-\lambda t_{n_0}^{\alpha})}{\alpha^{-1}E_{\alpha,\alpha}(-\lambda (t_{n_0}+\tau)^{\alpha})}.
\]
The right hand side can be bounded easily. In fact, the function (see \cite{mainardi2007probability})
\[
z\mapsto \alpha^{-1}E_{\alpha,\alpha}(-z)
=\sum_{k=0}^{\infty}\frac{(k+1)(-1)^k z^k}{\Gamma((k+1)\alpha+1)}
=E_{\alpha}'(-z)
=\int_0^{\infty}M_{\alpha}(r)re^{-rz}\,dr
\]
 is positive, continuous and nonincreasing in $z>0$ ( $M_{\alpha}(\cdot)$ is nonnegative). 
Moreover, it is continuous in $\alpha$ on $[1/2, 1]$. Hence, there exist $0<c_1<c_2$ such that for all $\alpha\in [1/2, 1)$
and all $z\le 3$, $c_1\le \alpha^{-1}E_{\alpha, \alpha}(z) \le c_2$.
Clearly, $\lambda(t_{n_0-1}+\tau_{n_0}+\tau)^{\alpha}\le 3^{\alpha} \lambda \max(t_{n_0}, \tau_{n_0}, \tau)^{\alpha}\le 3^{\alpha}$. Then, $\lambda (t_{n_0}+\tau)^{\alpha}\in (0, 3]$ and
$ w(t)/w(t+\tau)\le w(t_{n_0})/w(t_{n_0}+\tau) \le 2^{1-\alpha}c_2/c_1$.
Hence, for $j\ge n_0$,
\[
\frac{1}{\tau_j}\int_{t_{j-1}}^{t_j}w(s)\,ds
\le \frac{2^{1-\alpha}c_2}{c_1}w(t_j).
\]

Next, for $j\le n_0-1$, one has
$c_1\lambda\alpha t^{\alpha-1} \le w(t) \le c_2 \lambda \alpha t^{\alpha-1}$.
It follows that
\[
\frac{1}{\tau_j}\int_{t_{j-1}}^{t_j}w(s)\,ds
\le \frac{1}{t_j}\int_{0}^{t_j}w(s)\,ds
\le \frac{c_2\lambda t_j^{\alpha-1}}{\alpha}\le \frac{c_2}{c_1 \alpha}
w(t_j).
\]
Combining these two cases, the conclusion then follows.
\end{proof}

\begin{remark}
The upper bound in $\lambda \tau_n^{\alpha}\le 1$ is not crucial. In fact, for fixed $\gamma>0$, if we restrict $\lambda \tau_n^{\alpha}\le \gamma$, then we can find corresponding constant $\sigma$. 
\end{remark}

One can then establish the following Gr\"onwall inequalities.
\begin{theorem}\label{thm:decayest}
Let $\lambda>0$ and $v_0>c/\lambda$. Consider variable step completely positive discretization and $\{v_n\}$ is a nonnegative sequence.  
\begin{enumerate}[(1)]
\item Suppose Assumption \ref{ass:gron1} holds. If 
$\cD_{\tau}^{\alpha}v_n \ge -\lambda v_n+c$,
then it holds that
\begin{gather}
v_n \ge (v_0-c/\lambda)E_{\alpha}(-\nu^{-1}\lambda t_n^{\alpha})+c/\lambda.
\end{gather}

\item  Suppose Assumption \ref{ass:cupper} holds. If 
$\cD_{\tau}^{\alpha}v_n \le -\lambda v_n+c$,
then it holds that
\begin{gather}
v_n \le \left(v_0-\frac{c}{\lambda}\right)E_{\alpha}\left(-\frac{\lambda(1-\alpha)}{\rho_1}t_n^{\alpha}\right)+\frac{c}{\lambda}.
\end{gather}
If moreover it holds that $\lambda \tau_n^{\alpha}\le 1$,  then for the universal constant $\sigma$, it holds that
\begin{gather}
v_n \le \left(v_0-\frac{c}{\lambda}\right)E_{\alpha}\left(-\frac{\lambda}{ \sigma \rho_1}t_n^{\alpha}\right)+\frac{c}{\lambda}.
\end{gather}
\end{enumerate}
\end{theorem}
\begin{proof}
(1)
Consider the auxiliary equation
\[
D_c^{\alpha}y^{\nu}=\nu^{-1}(-\lambda y^{\nu}+c), y^{\nu}(0)=y_0.
\]
Now that $v_0>c/\lambda$, $y^{\nu}$ is a nonincreasing, convex function. 
Then, one may apply Lemma \ref{lmm:concavecomp} to $-y^{\nu}$ to obtain
\[
\cD_{\tau}^{\alpha}y^{\nu}\le \nu D^{\alpha}y^{\nu}(t_n)
=-\lambda y^{\nu}+c.
\]
The comparison principle in Theorem \ref{thm:implicitcp} then implies
$v_n\ge y^{\nu}(t_n)$, giving the desired result.

(2)  Consider the auxiliary problem
\[
D_c^{\alpha}y^{\rho}=\rho^{-1}(-\lambda y^{\rho}+c), y^{\rho}(0)=y_0.
\]
By Lemma \ref{lmm:firstcontrolrho1}, one can take $\rho=\rho_1/(1-\alpha)$ and then
\[
\cD_{\tau}^{\alpha}y^{\rho}\ge -\lambda y^{\rho}+c.
\]
Then, by Theorem \ref{thm:implicitcp}, one concludes that
$y(t_n)\le y^{\rho}(t_n)$.

If the discretization satisfies $\lambda\tau_n^{\alpha}\le 1$, motivated by Lemma \ref{lmm:firstcontrolrho1} and Lemma \ref{lmm:upperbound}, one can take for all $\alpha\in (0, 1)$ that
\[
\rho=\rho_1\frac{2c_2}{c_1}=\rho_1 \sigma.
\]
Note that for $\alpha\le 1/2$,  $\rho_1/(1-\alpha)< \rho_1\sigma$ so taking $\rho=\rho_1\sigma$ also works for $\alpha\le 1/2$. Now, since $\rho>1$, $\rho^{-1}\lambda \tau_n^{\alpha}\le 1$ holds, Lemma \ref{lmm:upperbound} applying to $y^{\rho}(\cdot)$ then gives
\[
\cD_{\tau}^{\alpha}y^{\rho}(t_n)\ge -\lambda y^{\rho}+c.
\]
The comparison principle in Theorem \ref{thm:implicitcp} then yields the last result.
\end{proof}

\begin{corollary}
If $g=0$, $v_n$ must decay to zero with upper bound as follows
$v_n \le v_0 E_{\alpha}\left(-\frac{\lambda}{\rho}t_n^{\alpha}\right)$,
where $\rho=\rho_1/(1-\alpha)$. If $\lambda \tau_n^{\alpha}\le 1$, one can take $\rho=\rho_1\sigma$.
\end{corollary}

\begin{remark}
For $L1$ scheme, the solution to the implicit scheme satisfies
\[
\left(v_0-\frac{c}{\lambda} \right)E_{\alpha}(-\lambda t_n^{\alpha})+\frac{c}{\lambda}
\le v_n \le \left(v_0-\frac{c}{\lambda}\right)E_{\alpha}\left(-\frac{\lambda}{\rho}t_n^{\alpha}\right)+\frac{c}{\lambda},
\]
where $\rho_1=1$ so that $\rho=\min(1/(1-\alpha),\sigma)$.
The lower bound holds for any discretization, while the upper bound holds for $\lambda\tau_n^{\alpha}\le 1$. This means that the solution of the implicit scheme is above the exact solution.
\end{remark}

\subsection{A Gr\"onwall inequality for growing linear functions}

We aim to establish a Gr\"onwall inequality for $f(u)=\lambda u+c$. Compared to \cite{liao2019discrete}, we aim to remove the requirement on the step size ratio. The key is again to show for some $\mu\in (0, 1]$ that
$\cD_{\tau}^{\alpha}y\ge \mu D^{\alpha}y$ where $y$ is the solution to the auxiliary equation.
We assume Assumption \ref{ass:gron1} and need
\[
c_{n-j}^n(y(t_j)-y(t_j-1))\ge  \frac{\mu}{\Gamma(1-\alpha)} \int_{t_{j-1}}^{t_j}(t_n-s)^{-\alpha}y'(s)\,ds.
\]
There are two special cases to guarantee this:
\begin{itemize}
\item If $y''(s)\le 0$ for $s\in [t_{j-1}, t_j]$, then under Assumption \ref{ass:gron1}, one can take $\mu=\nu$.
\item If $c_{n-j}^n \ge \mu_1 (t_n-s)^{-\alpha}$
or $\tau_j^{-1}(y(t_j)-y(t_{j-1}))\ge \mu_1 y'(s)$ for $s\in [t_{j-1}, t_j]$, one can then take $\mu=\nu \mu_1$. 
\end{itemize}
The function $y(s)$ is not concave unless $\lambda=0$. It is concave only near $t=0$. In fact, using $y'(s)=\lambda s^{\alpha-1}E_{\alpha,\alpha}(\lambda s^{\alpha})$ and the power series of Mittag-Leffler function in \eqref{eq:powerseriesofml},  one has the following claim.
\begin{lemma}
Let $w(t)=y'(t)$ where $y(t)=E_{\alpha}(\lambda t^{\alpha})$, $\lambda>0$ and $\alpha\in (0, 1)$. Then, there exists $t_*>0$ such that $w$ is decreasing on $(0, t_*)$ while increasing on $(t_*,\infty)$.
\end{lemma}
\begin{proof}
Since $w(t)=\lambda t^{\alpha-1}E_{\alpha,\alpha}(\lambda t^{\alpha})$, we set $z=\lambda t^{\alpha}$. Then,
\[
w=\lambda^{1/\alpha}z^{(\alpha-1)/\alpha}E_{\alpha, \alpha}(z)=:F(z),
\]
Recalling the power series of $E_{\alpha,\alpha}(z)$ in \eqref{eq:powerseriesofml}, one has
\begin{gather*}
\lambda^{-1/\alpha}F'(z)=(1-\alpha^{-1})z^{-\alpha^{-1}}E_{\alpha, \alpha}(z)
+z^{1-\alpha^{-1}}E_{\alpha,\alpha}'(z)
=z^{-\alpha^{-1}} \sum_{k=0}^{\infty}\frac{(1-\alpha^{-1}+k) z^k}{\Gamma(\alpha (k+1))}.
\end{gather*}
Let $k_0$ is the integer such that $1-\alpha^{-1}+k_0\le 0$ while $1-\alpha^{-1}+k_0+1>0$.
Such $k_0$ exists and $k_0\ge 0$. 
Define
\[
A(z)=\sum_{k=0}^{k_0}\frac{(\alpha^{-1}-k-1) z^k}{\Gamma(\alpha (k+1))},
\quad B(z)=\sum_{k=k_0+1}^{\infty}\frac{(1-\alpha^{-1}+k) z^k}{\Gamma(\alpha (k+1))}.
\]
Let $z_*$ be the first point such that $A(z_*)=B(z_*)$.  It is clear that $z_*$ exists and $z_*>0$. For $z<z_*$, $B(z)-A(z)<0$. For any $z>z_*$. Let $\theta=z/z_*>1$.
Then, it is clear that $A(z)\le \theta^{k_0}A(z_*)$ while $B(z)>\theta^{k_0}B(z_*)$. Then, $B(z)-A(z)
>\theta^{k_0}(B(z_*)-A(z_*))=0$. This means that there is only one point $z_*$ such that $F'(z)$ is zero and thus only one $t_*$ such that $w'(t_*)=0$.
\end{proof}
Then, for $t_n<t_*$, one can use the concavity as in Lemma \ref{lmm:concavecomp}. For large $t_n$, we turn to the second case above. Clearly,
$c_{n-j}^n \ge \mu_1 (t_n-s)^{-\alpha}$ for $s\in [t_{j-1}, t_j]$ cannot hold if $s$ is near $t_n$. Hence, we seek a lower bound for $y'(t)/y'(t+\tau)$.  By the integral representation \eqref{eq:mlfuncintegral}, one finds that
\begin{gather}\label{eq:integralrepcont}
E_{\alpha, \beta}(\lambda t^{\alpha})=
\frac{1}{\alpha}\lambda^{(1-\beta)/\alpha}t^{(1-\beta)}e^{\lambda^{1/\alpha}t}+\frac{1}{2\pi i\alpha}\int_{\gamma(\epsilon;\delta)} \frac{e^{\zeta^{1/\alpha}}\zeta^{(1-\beta)/\alpha}}{\zeta-\lambda t^{\alpha}}d\zeta.
\end{gather}
The first term grows exponentially while the second term goes to zero algebraically 
as $t\to\infty$. This observation gives a way to control $y'(t)/y'(t+\tau)$ for $t$ large. 

\begin{lemma}\label{lmm:lowerbdgrow}
Let $u(t)=E_{\alpha}(\lambda t^{\alpha})$ and $w(t):=u'(t)=\lambda t^{\alpha-1}E_{\alpha, \alpha}(\lambda t^{\alpha})$.  For any $0\le t<t+\tau\le T$ with $\lambda \tau^{\alpha}\le  1$, there is a universal constant $\mu_1 \in (0, 1)$ (independent of all parameters) such that
\[
\frac{1}{\tau}\int_{t}^{t+\tau}(T-s)^{-\alpha}\,ds \int_{t}^{t+\tau}w(s)\,ds\ge \mu_1 \int_{t}^{t+\tau}(T-s)^{-\alpha}w(s)\,ds.
\]
\end{lemma}

\begin{proof}
Note that $w$ decreases first and then increases.
Let $t_*$ be the transition point. We aim to find $\mu_1\in (0, 1)$ such that when $t_*\le s<s+\tau'$ with $\tau'\le \tau$ such that
\begin{gather}\label{eq:ratioaux}
\frac{w(s)}{w(s+\tau')}\ge \mu_1.
\end{gather}
We first prove the claim by assuming \eqref{eq:ratioaux}. In fact, 
\begin{itemize}
\item If $t+\tau\le t_*$, then $w$ is nonincreasing on this interval and thus,
\[
\frac{1}{\tau}\int_{t}^{t+\tau}(T-s)^{-\alpha}\,ds \int_{t}^{t+\tau}w(s)\,ds\ge  \int_{t}^{t+\tau}(T-s)^{-\alpha}w(s)\,ds.
\]

\item If $t\ge t_*$, then by \eqref{eq:ratioaux}, one has
\[
\frac{1}{\tau}\int_{t}^{t+\tau}(T-s)^{-\alpha}\,ds \int_{t}^{t+\tau}w(s)\,ds 
\ge \mu_1 \int_t^{t+\tau}(T-s)^{-\alpha}w(t+\tau)\,ds.
\]

\item For $t_*\in (t, t+\tau)$, define
\[
\tilde{w}(s)=\begin{cases}
w(s), \quad s\le t_*,\\
w(t_*) \quad s>t_*
\end{cases}
\]
Then, $\tilde{w}$ is nonincreasing so one has by Chebyshev's sorting inequality that
\[
\begin{split}
\frac{1}{\tau}\int_{t}^{t+\tau}(T-s)^{-\alpha}\,ds  \int_{t}^{t+\tau}w(s)\,ds
&\ge \frac{1}{\tau} \int_{t}^{t+\tau}(T-s)^{-\alpha}\,ds  \int_{t}^{t+\tau}\tilde{w}(s)\,ds\\
&\ge  \int_{t}^{t+\tau}(T-s)^{-\alpha}\tilde{w}(s)\,ds.
\end{split}
\]
For $s\ge t_*$, $\tilde{w}(s)=w(t_*)\ge \mu_1 w(s)$. Hence, the claim still holds for this case.
\end{itemize}

Next, we establish \eqref{eq:ratioaux}. Letting $z=\lambda t^{\alpha}$ and taking $\delta=\alpha\pi$ in \eqref{eq:integralrepcont}, for $\alpha\in (0, 1)$ and $\beta=\alpha<1+\alpha$, one can take $\e\to 0$ to have
\[
\begin{split}
E_{\alpha,\alpha}(z)&=\frac{1}{\alpha}z^{(1-\alpha)/\alpha}e^{z^{1/\alpha}}+\frac{1}{\pi \alpha}\int_0^{\infty}r^{(1-\alpha)/\alpha}e^{-r^{1/\alpha}}
\frac{r\sin(\pi(1-\alpha))}{r^2-2rz\cos(\pi \alpha)+z^2}dr \\
&=:I_1+I_2.
\end{split}
\]
Using $\sin(\pi\alpha/2)\ge \alpha$ and  the simple bound $e^{-r^{1/\alpha}}r^{1/\alpha}\le e^{-1}$, we find
\[
I_2 \le \frac{\sin(\pi \alpha)}{\pi \alpha}\int_{0}^{\infty}\frac{e^{-r^{1/\alpha}}r^{1/\alpha}}{(r-z)^2+2z^2\alpha^2}dr  \le \frac{\sin(\pi \alpha)}{\pi \alpha}\frac{\pi}{\sqrt{2}z\alpha e}.
\]
Hence, one has
\begin{gather}\label{eq:aux2}
\frac{I_2(z)}{I_1(z)}
\le \frac{\pi/\sqrt{2}}{z^{1/\alpha}e^{1+z^{1/\alpha}}}.
\end{gather}

Below, we consider two cases.

{\bf Case 1: $\alpha\ge 1/2$}

For $z\ge 1$, one has $I_2(z)/I_1(z)\le \pi/(\sqrt{2}e^2)$.
Decompose 
\[
w(t)=\lambda t^{\alpha-1}I_1(\lambda t^{\alpha})
+\lambda t^{\alpha-1}I_2(\lambda t^{\alpha})=:w_1(t)+w_2(t).
\]
Then, $w_1(t)=\alpha^{-1}\lambda^{1/\alpha}e^{\lambda^{1/\alpha}t}$, and $w_2/w_1\le \frac{\pi}{\sqrt{2}e^2}$, which implies that
\[
\frac{w(t)}{w(t+\tau)}\ge \frac{w_1(t)}{w_1(t+\tau)+w_2(t+\tau)}\ge \frac{1}{1+\frac{\pi}{\sqrt{2}e^2}}\frac{w_1(t)}{w_1(t+\tau)}
\ge  \frac{1}{1+\frac{\pi}{\sqrt{2}e^2}}e^{-1}.
\]

For $z\le 1$, one has
\[
\alpha^{-1}E_{\alpha, \alpha}(z)=\sum_{k=0}^{\infty}\frac{z^k}{\alpha \Gamma((k+1)\alpha)}
=\sum_{k=0}^{\infty}\frac{(k+1)z^k}{\Gamma((k+1)\alpha+1)}.
\] 
Since $4/5<\Gamma(1+\alpha)<1$ and the $\alpha\mapsto \Gamma((k+1)\alpha+1)$ is increasing since $(k+1)\alpha+1\ge 1.5$, one finds that $\alpha^{-1}E_{\alpha, \alpha}(z)$ is uniformly bounded as
\[
1\le \alpha^{-1}E_{\alpha, \alpha}(z) \le 2E_{1/2, 1/2}(1)=:c'.
\]
Then, one finds that $\alpha\lambda t^{\alpha-1}\le w(t)\le \alpha\lambda t^{\alpha-1} c'$.

Combining these two results, if $t_*$ is the transition point for $w$, then for all $t_*\le s\le s+\tau'$ with 
$\lambda (\tau')^{\alpha}\le 1$:
\[
\frac{w(s)}{w(s+\tau')}\ge \frac{1}{c'}\frac{1}{1+\frac{\pi}{\sqrt{2}e^2}}e^{-1}=:\mu_{1,1}.
\]

{\bf Case 2: $\alpha\le 1/2$}

We take $M>1$ to be determined later.  For $z\ge M^{-\alpha}$, one has by \eqref{eq:aux2} that
\[
\frac{I_2(z)}{I_1(z)}\le \frac{M\pi}{\sqrt{2}e^{1+1/M}}.
\]
Using similar argument as above, if $\lambda t^{\alpha} \ge M^{-\alpha}$, one has
\[
\frac{w(t)}{w(t+\tau)}\ge \frac{1}{1+\frac{M\pi}{\sqrt{2}e^{1+1/M}}}e^{-1}=:\mu_{1,2}.
\]

We show that $w$ is monotone for $\lambda t^{\alpha}< M^{-\alpha}$ if $M$ is large so that $\lambda t_*^{\alpha}\ge M^{-\alpha}$. Recall
\[
w(t)=F(\lambda t^{\alpha}), \quad F(z)=\lambda^{1/\alpha}z^{(\alpha-1)/\alpha}E_{\alpha, \alpha}(z).
\]
 Since $dz/dt>0$, one only needs to consider:
\[
\frac{d}{dz}\log F=\frac{\alpha-1}{\alpha z}+\frac{F_1'(z)}{F_1(z)}, 
\quad F_1(z)=\alpha^{-1}E_{\alpha, \alpha}(z).
\]
Noting that $\inf_{a>0}\Gamma(a)>0.8$, one has
\[
F_1'(z)=\sum_{k=1}^{\infty}\frac{(k+1)k z^{k-1}}{\Gamma((k+1)\alpha+1)}
\le \frac{5}{4}\sum_{k=1}^{\infty} (k+1)k z^{k-1}
=\frac{5}{4}\frac{1}{(1-z)^3}.
\]
Moreover, since $\Gamma(a)<1$ for $a\in (1, 2)$, one has
\[
F_1(z)=\sum_{k=0}^{\infty}\frac{(k+1)z^k}{\Gamma((k+1)\alpha+1)}
\ge \sum_{k=0}^{[1/\alpha]}(k+1)z^k
=\frac{d}{dz}\frac{1-z^{[1/\alpha]+2}}{1-z}\ge \frac{d}{dz}\frac{1-z^{1/\alpha}}{1-z}.
\]
Here, we used the fact that $\frac{d}{dz}\frac{z^m}{1-z}=\frac{z^{m-1}[m(1-z)+z]}{(1-z)^2}$ is decreasing in $m$ for $m>1$. In fact, the derivative of $m\mapsto \log(z^{m-1}[m(1-z)+z])$ is negative. These imply that
\[
\frac{F_1'}{F_1}\le \frac{5}{2}\frac{1}{(1-z)(1-(1-z+\alpha z)z^{1/\alpha}/(\alpha z))}.
\]
Consider the expression
\[
A:=(1-z+\alpha z)z^{1/\alpha}/(\alpha z))=(1-z)\frac{z^{1/\alpha}}{\alpha z}+z^{1/\alpha}
<(1-z)\frac{z^{1/\alpha}}{\alpha z}+M^{-1}.
\]
Optimizing over $\alpha\in (0, 1/2)$, one has
\[
\frac{z^{1/\alpha}}{\alpha}
\le\begin{cases}
2z^2, & z\le 1/\sqrt{e},\\
\frac{1}{-\ln(z)e}, & z>1/\sqrt{e}.
\end{cases}
\]
Then, one finds that $A<1/2+1/M$ (note that $e^{-1}(z^{-1}-1)/\ln(1+(z^{-1}-1))$ is monotone in the case $z>1/\sqrt{e}$). Hence,
\[
\frac{F_1'}{F_1}\le \frac{5}{1-2M^{-1}}\frac{1}{1-z}.
\]
Moreover, $1-M^{-\alpha}=\alpha (\ln M)M^{-\xi}>\alpha (\ln M)M^{-\alpha}>\alpha (\ln M) z$.
This then indicates that
\[
\frac{d}{dz}\log F<\frac{\alpha-1+5/(\ln M(1-2M^{-1}))}{\alpha z}, \quad \alpha\le 1/2.
\]
Hence, for $M$ large enough, this is negative. This means that $\lambda t_*^{\alpha}>M^{-\alpha}$. 
Taking $\mu_1=\min(\mu_{1,1}, \mu_{1,2})$ clearly suffices.
\end{proof}

\begin{theorem}\label{thm:grongrow}
Suppose Assumption \ref{ass:gron1} holds. Let $\mu:=\nu \mu_1$ where $\mu_1$ is the constant in Lemma \ref{lmm:lowerbdgrow}.
Assume $\lambda \tau_n^{\alpha}<\min(\mu, \nu/\Gamma(2-\alpha))$.  Suppose a nonnegative sequence $y_n$ satisfies
\[
\cD_{\tau}^{\alpha}y_n \le \lambda y_n+c, \lambda>0,
\]
then one has
\[
y_n \le \left(y_0+\frac{c}{\lambda}\right)E_{\alpha}(\mu^{-1}\lambda t_n^{\alpha})-\frac{c}{\lambda}.
\]
\end{theorem}

\begin{proof}
Consider the equation
\[
D^{\alpha}y^{\mu}=\mu^{-1}(\lambda y^{\mu}+c), y^{\mu}(0)=y_0.
\]
Note that $\mu^{-1}\lambda \tau_n^{\alpha}<1$. By Lemma \ref{lmm:lowerbdgrow}, one has
\[
\cD_{\tau}^{\alpha}y^{\mu}(t_n)
\ge \lambda y^{\mu}+c.
\]

If $\lambda\tau_n^{\alpha}<\nu/\Gamma(2-\alpha)$, one then has
$c_0^n>M$.  By the comparison principle in Theorem \ref{thm:implicitcp} and formula \eqref{eq:exactsol}, one has $y_n \le y^{\mu}(t_n)$, which is the desired result.
\end{proof}

\section{Applications to dissipative systems}\label{sec:application}

In this section, we consider two examples for the dissipative systems to illustrate how the Gr\"onwall inequalities above can be applied. The first is the standard subdiffusion equation while the second is the time fractional Allen-Cahn equation. 

\subsection{Example 1: subdiffusion equation}

Consider the following subdiffusion equation
\begin{gather}\label{eq:subdiffeqn}
\begin{split}
& D_c^{\alpha}u=\Delta u+f(x),\quad x\in \Omega \\
& u|_{\partial\Omega}=0, \quad u(x,0)=u_0(x).
\end{split}
\end{gather}

Consider the approximation where the time derivative is discretized by the L1 scheme on nonuniform mesh  (i.e., \eqref{eq:L1a}-\eqref{eq:L1b})
and the Laplacian operator is discretized by the centered difference method. Then, one has
\begin{gather}
\cD_{\tau}^{\alpha}u_n=\Delta_h u_n+f(x).
\end{gather}
The truncation error for spatial derivative is clearly $O(h^2)$.
Regarding the truncation error for time, if the solution is assumed to be smooth, the truncation error for $L1$ scheme is $\Delta t^{2-\alpha}$ \cite{lin2007finite}. However, taking into account the singularity near $t=0$, the truncation error reduces to $\Delta t$. Using graded mesh, $t_n=T(n/N)^{r}$ can improve the accuracy \cite{stynes2017error}. Here, we consider the truncation error on a general nonuniform mesh
\[
R_n:=\cD_{\tau}^{\alpha}u(\cdot, t_n)-D_c^{\alpha}u(\cdot, t_n)
=\cD_{\tau}^{\alpha}u(\cdot, t_n)-(\Delta_h u(\cdot, t_n)+f(\cdot)).
\] 
Then, the truncation error is bounded by
\begin{gather}
\|R_n\|_{\ell^2} \le C(\tau+h^{2}),
\end{gather}
where $\tau=\max_i \tau_i$, $C$ is independent of $n$. 

Let $\Omega_h$ be the set of spatial grid points and consider 
\[
\langle u, v\rangle_{\Omega_h}:=\sum_{x\in \Omega_h} u(x)v(x) h^d,\quad \|u\|_{\ell^2}^2:=\langle u, u\rangle_{\Omega_h}.
\]
For the discrete Laplacian, there exists a constant $\kappa>0$ such that for all discrete functions $v$ being zero on $\partial\Omega_h$, one has
\[
-\langle v, \Delta_h v\rangle_{\Omega_h}\ge \kappa \|v\|_{\ell^2}^2. 
\]
We have the following conclusions.
\begin{proposition}\label{pro:subdiffusion}
Let $u_{\infty}$ be the steady solution of \eqref{eq:subdiffeqn} and $u_{\infty}^h$ be the steady solution of the numerical scheme. Then,
if $\kappa \tau_n^{\alpha}\le 1$, one has
\begin{gather}
\|u_n-u_{\infty}^h\|_{\ell^2} \le \|u_0-u_{\infty}^h\|_{\ell^2} E_{\alpha}(-\sigma^{-1}\kappa t_n^{\alpha}).
\end{gather}
Moreover, the error satisfies
\[
\sup_n \|u_n-u(t_n)\|_{\ell^2}\le \frac{C}{\kappa}(\tau+h^2)(1-E_{\alpha}(-\kappa t_n^{\alpha})).
\]
Consequently, $\|u_{\infty}-u_{\infty}^h\|_{\ell^2}\le C(\tau+h^2)$.
\end{proposition}

\begin{proof}
Since $u_{\infty}^h$ is a steady solution to the numerical scheme, one has
\[
\cD_{\tau}^{\alpha}(u_n-u_{\infty}^h)=\Delta_h(u_n-u_{\infty}^h).
\]
Then, by Lemma \ref{lmm:convexfunctional},
\[
\cD_{\tau}^{\alpha}\|u_n-u_{\infty}^h\|_{\ell^2}
\le  \left\langle \frac{u_n-u_{\infty}^h}{\|u_n-u_{\infty}^h\|_{\ell^2}},
\cD_{\tau}^{\alpha}(u_n-u_{\infty}^h)  \right\rangle
\le -\kappa \|u_n-u_{\infty}^h\|_{\ell^2}.
\]
Theorem \ref{thm:decayest} then gives the desired result (Assumption \ref{ass:cupper} holds with $\rho_1=1$ for L1 discretization).

By the definition of truncation error,
\[
\cD_{\tau}^{\alpha}u(\cdot, t_n)=\Delta_h u(\cdot, t_n)+f(\cdot)+R_n.
\]
If one defines the error $e_n:=u_n-u(\cdot, t_n)$, one has
\[
\cD_{\tau}^{\alpha}e_n=\Delta_h e_n-R_n.
\]
Pairing with $\frac{e_n}{\|e_n\|}$, one has
\[
\cD_{\tau}^{\alpha}\|e_n\|_{\ell^2}\le -\kappa \|e_n\|_{\ell^2}
+C(\tau+h^2).
\]
Applying Theorem \ref{thm:unibound} (with $v_0=0$ and $\nu=1$) gives the desired control to the error.
\end{proof}

\subsection{Example 2: time fractional Allen-Cahn equation}

We consider the following 1D time fractional Allen-Cahn equation as an example \cite{li2021sharp}. 
\begin{gather}
\begin{split}
&D_c^{\alpha} u=\kappa^2\partial_{xx}u+(u-u^3),\quad x\in \mathbb{T}, \\
& u|_{t=0}=u_0.
\end{split}
\end{gather}
Here, $\mathbb{T}$ is the 1D torus with length $2\pi$ (i.e., $[-\pi, \pi)$ with periodic boundary condition). The equation is associated with a free energy
\[
E(u)=\int_{\T} \frac{\kappa^2}{2} |\partial_x u|^2+\frac{1}{4}(u^2-1)^2\,dx ,
\]
and the equation is actually the time fractional gradient flow of this free energy in $L^2$ 
\[
D_c^{\alpha} u=-\frac{\delta E}{\delta u}.
\]
See \cite{li2019discretization} for some discussion on how one uses the discretization to analyze the behaviors of time fractional gradient flows.

Below, we consider $\kappa>1$ and the discretization with $L1$ scheme:
\begin{gather*}
\cD_{\tau}^{\alpha}u_n=\kappa^2D^2u_n+(u_n-u_n^3).
\end{gather*}
Here, $D^2$ means that the spatial derivative is discretized by the Fourier spectral method. For any $p>0$, there exists $C>0$ such that the truncation error satisfies
\[
r_n:=\|\cD_{\tau}^{\alpha}u(\cdot, t_n)-(\kappa^2\partial_{xx}u_n+(u_n-u_n^3))  \|_{\ell^2}  \le C(\tau+h^p).
\]
The truncation error for the time discretization has been discussed above in the first example. The spatial truncation error is standard for spectral method. For spectral discretization, one has
\[
-\langle v, D^2v\rangle \ge \|v\|_{\ell^2}^2, \quad \forall v\quad \text{with}\quad \hat{v}_0=0.
\]

\begin{proposition}
Suppose that $u_0$ is an odd function on $[-\pi, \pi]$ and $\kappa^2>1$. Assume also that $(\kappa^2-1)\tau_n^{\alpha}\le 1$, then the $L^2$ norm of the numerical solution satisfies
\[
\|u_n\|_{\ell^2}\le C E_{\alpha}(-\sigma^{-1}(\kappa^2-1)t_n^{\alpha}) \sim Ct_n^{-\alpha}, \quad n\to\infty.
\]
The error satisfies
\[
\|u(t_n)-u_n\|_{\ell^2}\le \frac{C(\tau+h^p)}{\kappa^2-1}\left(1-E_{\alpha}(-(\kappa^2-1)t_n^{\alpha})\right).
\]
\end{proposition}
Note that if $u_0$ is an odd function, then the zero Fourier mode of $u_n$
preserves to be zero. Hence, one has
\[
\begin{split}
\left\langle \frac{u_n}{\|u_n\|_{\ell^2}}, \kappa^2 D^2 u_n
+(u_n-u_n^3)\right\rangle  &\le -\kappa^2 \|u_n\|_{\ell^2}+\|u_n\|_{\ell^2}
-\frac{1}{\|u_n\|_{\ell^2}}\|u_n\|_{\ell^4}^4 \\
& \le -(\kappa^2-1)\|u_n\|_{\ell^2}.
\end{split}
\]
The detailed proof would be the same as that for Proposition \ref{pro:subdiffusion} and we omit them.

\section*{Acknowledgement}

This work was financially supported by the National Key R\&D Program of China, Project Number 2021YFA1002800 and 2020YFA0712000. The work of Y. Feng was partially sponsored by NSFC 12301283, Shanghai Sailing program 23YF1410300 and Science and Technology Commission of Shanghai Municipality (No. 22DZ2229014). The work of L. Li was partially supported by NSFC 12371400 and 12031013,  Shanghai Science and Technology Commission (Grant No. 21JC1403700, 20JC144100), the Strategic Priority Research Program of Chinese Academy of Sciences, Grant No. XDA25010403. The work of J.-G. Liu was partially supported by NSF DMS-2106988.
The work of T. Tang was partially supported by Science Challenge Project (Grant No. TZ2018001) and NSFC 11731006 and K20911001.

\bibliographystyle{plain}
\bibliography{frac}

\end{document}